\documentclass[preprint,11pt,a4paper,nopreprintline]{elsarticle}
\usepackage[T1]{fontenc}
\usepackage[utf8]{inputenc}
\usepackage{lmodern}
\usepackage[margin=27mm]{geometry}
\usepackage{amsmath,amssymb,amsthm,mathtools}
\usepackage{microtype}
\usepackage[hidelinks]{hyperref}
\numberwithin{equation}{section}
\newtheorem{theorem}{Theorem}[section]
\newtheorem{proposition}[theorem]{Proposition}
\newtheorem{lemma}[theorem]{Lemma}
\theoremstyle{remark}
\newtheorem{remark}[theorem]{Remark}
\newcommand{\R}{\mathbb R}
\newcommand{\Div}{\operatorname{div}}
\newcommand{\dd}{\,\mathrm{d}}
\newcommand{\PP}{\mathbb P}
\newcommand{\QQ}{\mathbb Q}
\allowdisplaybreaks[2]

\begin{document}
\begin{frontmatter}

\title{Global solutions to the compressible\\
Navier--Stokes--Poisson equations\\
in the half-space and exterior domains}

\author[inst1]{Hairong Liu}
\ead{hrliu@njust.edu.cn}
\author[inst1]{Jinlong Wu\corref{cor1}}
\ead{jinlongno@njust.edu.cn}
\author[inst2]{Hua Zhong}
\ead{huazhong@swjtu.edu.cn}
\cortext[cor1]{Corresponding author.}

\address[inst1]{School of Mathematics and Statistics,
Nanjing University of Science and Technology,\\
Nanjing 210094, P.R. China}
\address[inst2]{School of Mathematics, Southwest Jiaotong University,\\
Chengdu 611756, P.R. China}

\begin{abstract}
We study the initial boundary value problem for the compressible
Navier--Stokes--Poisson equations without damping in the three-dimensional
half-space and in smooth exterior domains. The velocity satisfies the
no-slip boundary condition, and the electric potential satisfies the
homogeneous Neumann boundary condition. For small perturbations of a
constant state in $H^3$, with a small initial electric field in $L^2$,
we establish global existence and uniqueness of strong solutions.
The uniform estimates combine density dissipation, tangential and normal
derivative estimates, and a modified energy expressed through the
Helmholtz projection. In exterior domains, boundary normal coordinates
and localized elliptic estimates control the geometric terms.
We also obtain algebraic decay in the full energy norm. In the half-space,
finite tangential negative Sobolev norms of order $-s$, with
$\frac12\leq s<1$, yield the rate $(1+t)^{-s/2}$. In exterior domains,
integrable initial density and velocity perturbations yield the same
form of estimate for $0<s<\frac14$. The exterior decay proof uses
whole-space negative-order estimates for localized variables and the
time integrability of the dissipation.
\end{abstract}


\end{frontmatter}

\section{Introduction and main results}\label{sec:introduction}

The compressible Navier--Stokes--Poisson system consists of the
Navier--Stokes equations coupled with the self-consistent Poisson equation.
It describes the motion of charged particles under an electrostatic force
and arises in hydrodynamic models of plasmas and semiconductors; see
\cite{HsiaoLi2010} for an introduction. In this paper, we consider the
isentropic Navier--Stokes--Poisson system in either the half-space
$\R^3_+=\R^2\times\R_+$ or a smooth connected exterior domain
$\R^3\setminus\overline{\mathcal O}$, where $\mathcal O$ is bounded.
We write $\Omega$ for the domain under consideration. In the half-space,
$x=(y,x_3)$ with $y=(x_1,x_2)$.
The equations take the form
\begin{equation}\label{eq:NSP}
\left\{
\begin{aligned}
&\rho_t+\Div(\rho u)=0,\\
&\rho(u_t+u\cdot\nabla u)+\nabla p(\rho)
  =\mu\Delta u+(\mu+\lambda)\nabla\Div u+\rho\nabla\phi,\\
&\Delta\phi=\rho-1,
\end{aligned}
\right.
\qquad x\in\Omega,\quad t>0,
\end{equation}
where $\rho$, $u=(u_1,u_2,u_3)$ and $\phi$ denote the density,
velocity and electric potential, respectively. The pressure is given
by $p(\rho)=\rho^\gamma$, with $\gamma\geq1$, and the viscosity
coefficients satisfy
\begin{equation}\label{eq:viscosity}
\mu>0,\qquad \lambda+\frac23\mu\geq0.
\end{equation}
The constant background charge density has been normalized to one.
There is no damping term in the momentum equation. We impose the
boundary and initial conditions
\begin{equation}\label{eq:data}
\left\{
\begin{aligned}
&u=0,\qquad \partial_n\phi=0\quad\text{on }\partial\Omega,\\
&(\rho,u)(0,x)=(\rho_0,u_0)(x),
\end{aligned}
\right.
\end{equation}
where $n$ is the outward unit normal. We consider solutions near the constant state $(\rho,u)=(1,0)$.
We use the homogeneous Sobolev space
\begin{equation}\label{eq:normalization}
\dot H^1(\Omega)
=\{f\in L^6(\Omega):\nabla f\in L^2(\Omega)\},
\qquad \|f\|_{\dot H^1}:=\|\nabla f\|_{L^2}.
\end{equation}
The electric potential is chosen in $\dot H^1(\Omega)$ at each time,
which fixes the additive constant in the Neumann problem.

The global existence and large-time behavior of solutions to the
compressible Navier--Stokes equations have been studied extensively.
For the Cauchy problem, Matsumura and
Nishida~\cite{MatsumuraNishida1979,MatsumuraNishida1980} established global
existence of smooth solutions when the initial data are close to a
constant equilibrium. Their energy method provides a basic approach to
the interaction between the hyperbolic equation for the density and the
parabolic equations for the velocity and temperature. The smallness
assumption controls the nonlinear terms, while the coupling of the
equations supplies estimates for the density, which has no direct
viscous dissipation. In a different functional setting,
Danchin~\cite{Danchin2000} proved global well-posedness near a constant
state in critical Besov spaces for the isentropic system.

Considerable progress has also been made for initial data with less
regularity or larger oscillations. Hoff~\cite{Hoff1995} constructed
global solutions for multidimensional isothermal compressible flow with
small discontinuous perturbations. For the three-dimensional barotropic
system in bounded domains, Feireisl, Novotn\'y and
Petzeltov\'a~\cite{FeireislNovotnyPetzeltova2001} obtained global
finite-energy weak solutions for isentropic pressure laws with
adiabatic exponent $\gamma>3/2$. Huang, Li and Xin~\cite{HuangLiXin2012} established
global classical solutions to the three-dimensional isentropic Cauchy
problem with small initial energy, allowing large oscillations and
vacuum. For smooth perturbations, the
large-time behavior gives further information on the dissipative
structure; in particular, Hoff and Zumbrun~\cite{HoffZumbrun1995}
studied the multidimensional diffusion waves of compressible flow.

For domains with boundaries, the estimates must also be compatible with
the boundary conditions. Matsumura and
Nishida~\cite{MatsumuraNishida1983} proved global existence of small
solutions to the initial-boundary value problem in the half-space and
in exterior domains. Their treatment of derivatives near the boundary
is an important part of the small-data theory. Decay estimates in
unbounded domains were studied by Deckelnick~\cite{Deckelnick1992}, and
the behavior in three-dimensional exterior domains was further
investigated by Kobayashi and Shibata~\cite{KobayashiShibata1999} and
Kobayashi~\cite{Kobayashi2002}. In the half-space, Kagei and
Kobayashi~\cite{KageiKobayashi2002,KageiKobayashi2005} obtained detailed
large-time and asymptotic results. Under
spherical symmetry, Jiang~\cite{Jiang1996} proved global existence of
smooth solutions for large initial data in the exterior of a ball.

The Navier--Stokes--Poisson system introduces a self-consistent force
into the equations of compressible viscous flow. The potential is
determined by an elliptic equation, and its interaction with the fluid
changes both the energy estimates and the large-time behavior. We
refer to Hsiao and Li~\cite{HsiaoLi2010} for a discussion of the
compressible Navier--Stokes--Poisson equations. Donatelli~\cite{Donatelli2003}
studied local and global existence for the coupled problem.
Kobayashi and Suzuki~\cite{KobayashiSuzuki2008} considered global weak
solutions in bounded domains for the gravitational Navier--Stokes--Poisson
system. 
For small perturbations of a constant state in the whole space,
Hao and Li~\cite{HaoLi2009} established global existence in hybrid Besov
spaces in three and higher dimensions. Li, Matsumura and
Zhang~\cite{LiMatsumuraZhang2010} obtained optimal decay rates for the
three-dimensional compressible Navier--Stokes--Poisson system. Their
analysis exhibits the influence of the electric field on the decay of
the momentum. Wang and Wu~\cite{WangWu2010} derived pointwise estimates
and optimal $L^p$ decay rates in several space dimensions. The
corresponding decay problem for the non-isentropic system in
$\mathbb R^3$ was considered by Zhang, Li and Zhu~\cite{ZhangLiZhu2011}.

There are also energy methods and approaches in critical spaces
for the Cauchy problem. Wang~\cite{Wang2012} obtained decay estimates by an energy
method involving negative Sobolev norms. Chikami and
Danchin~\cite{ChikamiDanchin2017} proved global existence and time decay
estimates in critical spaces for the system with a repulsive potential,
in space dimensions greater than or equal to two. Bie, Wang and
Yao~\cite{BieWangYao2017} studied optimal decay in the critical $L^p$
framework. 
Global solutions near nonconstant equilibria have also attracted much
attention. Tan, Wang and Wang~\cite{TanWangWang2015} studied stability
of steady states associated with a non-flat doping profile and obtained
decay estimates under additional assumptions. In one space dimension,
Tan, Yang, Zhao and Zou~\cite{TanYangZhaoZou2013} established global
solutions to a non-isentropic Navier--Stokes--Poisson system with large
initial data. Their result allows density-dependent viscosity and
thermal conductivity depending on the density and temperature. 
For the initial-boundary value problem, Liu and
Zhong~\cite{LiuZhong2021} proved global existence and exponential
stability of smooth solutions near steady states in three-dimensional
bounded domains. Chen, Huang and Shi~\cite{ChenHuangShi2024} established
global well-posedness of classical solutions in three-dimensional
bounded domains with slip boundary conditions and small initial energy,
allowing vacuum and large oscillations. 
The work most closely related to the present paper is that of Liu,
Luo and Zhong~\cite{LiuLuoZhong2020}, who studied the
Navier--Stokes--Poisson and Euler--Poisson equations of plasma in
exterior domains. For the Navier--Stokes--Poisson system outside a ball,
they obtained global solutions with large spherically symmetric data;
this result includes the case without damping. For general
three-dimensional exterior domains, they proved global existence and
exponential decay for small perturbations in the presence of positive
damping. Their proof combines energy estimates, tangential and normal
derivative estimates, and Stokes estimates in which the pressure and
electric potential are treated together.

Du and Liu~\cite{DuLiu2025} proved global stability near nonconstant
steady states in the exterior of a ball with Navier-slip velocity and
Neumann potential conditions. Their work allows large variations of the
doping profile and contains refined exterior Neumann estimates. Here the
velocity satisfies the homogeneous Dirichlet condition, and the estimates
are established near a constant state in smooth exterior domains.

Recent work has also addressed Poisson-coupled flow in the half-space.
Wu and Wang~\cite{WuWang2023} obtained global small solutions for the
compressible elastic Navier--Stokes--Poisson system with damping in
half-spaces and exterior domains, allowing Dirichlet or Neumann
conditions for the potential. Shen, Wang and Wu~\cite{ShenWangWu2025}
studied the damped elastic system with Hodge boundary conditions for the
velocity. For elastic flows without damping, Shen and
Wang~\cite{ShenWang2026a,ShenWang2026b} established global small
solutions in the half-space under Hodge velocity and Neumann potential
conditions, treating both the isentropic and heat-conducting systems.
Their estimates use the deformation and an effective electric field.
For the Navier--Stokes--Poisson system with damping, Gao, Peng, Wei and
Yao~\cite{GaoPengWeiYao2026} proved global uniform regularity and a
vanishing viscosity limit under slip boundary conditions.
In the present paper, we consider the half-space and exterior problems for
\eqref{eq:NSP} with the no-slip condition for the velocity and the
homogeneous Neumann condition for the potential. We establish global
existence of small strong solutions without a symmetry assumption
or a damping term.

The absence of damping alters the structure of the energy estimates.
The basic energy
identity gives the dissipation of $\nabla u$, whereas the $L^2$ norms of
$u$ and $\nabla\phi$ are controlled only as part of the energy. In these
unbounded domains, these norms cannot be estimated by a Poincar\'e
inequality. To obtain the density dissipation, we test the momentum
equation by the density gradient and use the Poisson equation and its
Neumann boundary condition. For higher derivatives, we first estimate
the time and tangential derivatives. The normal component of the
momentum equation then provides estimates for the normal derivatives
of the density, and the Stokes estimates recover the full spatial
regularity of the velocity.

A further difficulty arises in the estimate for $\nabla u_t$.
The Neumann condition on $\phi_t$ does not remove the tangential
electric contribution at the boundary when the differentiated momentum
equation is tested by the viscous operator applied to $u_t$. We use the
Helmholtz projection to write this contribution as a time derivative
and volume integrals. The resulting correction is bounded in terms of the
lower-order energy and the density estimates. Combining these bounds
gives a closed energy inequality, with constants independent of time,
and hence a unique global strong solution near the constant state.

In an exterior domain, we use boundary normal coordinates to estimate
the tangential and normal derivatives. The pure second normal velocity
derivative cancels in the normal component of
$\Delta u-\nabla\Div u$. Local Stokes estimates then recover the mixed
derivatives in an order determined by the number of normal derivatives.
The geometric commutators are controlled by interpolation and lower-order
dissipation. Refined Neumann estimates bound derivatives of the electric
field by the density.

We obtain algebraic decay under additional assumptions on the initial
low frequencies. In the half-space we use tangential negative Sobolev
norms and the formulation in terms of the momentum $\rho u$, following
the negative-order energy approach of Wang~\cite{Wang2012}.
For the exterior problem, we localize the momentum and the potential
away from the obstacle while preserving the Poisson equation. The
resulting errors are supported in a fixed annulus. Whole-space Sobolev
inequalities bound their contribution by the integrable dissipation,
which gives uniform negative-order estimates and algebraic decay.

We write $q=\rho-1$ and $q_0=\rho_0-1$. Throughout the paper,
$\|\cdot\|$ and $(\cdot,\cdot)$ denote the norm and inner product
in $L^2(\Omega)$. The notation $\|\cdot\|_{L^p}$ is used for
other Lebesgue norms. The squared norm of a tuple is the sum of the
squared norms of its components. All Sobolev spaces are over $\Omega$ unless
otherwise specified, and vector-valued spaces are understood when
appropriate. The letters $c$ and $C$ denote positive constants,
which may change from line to line and depend only on
$\mu$, $\lambda$, $\gamma$ and, in the exterior case, $\Omega$, unless another dependence is stated.

\begin{theorem}\label{thm:global}
Let $\Omega=\R^3_+$. Assume \eqref{eq:viscosity} and $\gamma\geq1$. Let
\begin{equation}\label{eq:initial-regularity}
q_0\in H^3(\Omega),\qquad
u_0\in H^3(\Omega)\cap H^1_0(\Omega).
\end{equation}
Let $\phi_0\in\dot H^1(\Omega)$ solve
\begin{equation}\label{eq:initial-poisson}
\Delta\phi_0=q_0\quad\text{in }\Omega,\qquad
\partial_n\phi_0=0\quad\text{on }\partial\Omega.
\end{equation}
Suppose that the initial data satisfy the compatibility condition
\begin{equation}\label{eq:compatibility}
\left.
\left[
\frac{\mu\Delta u_0+(\mu+\lambda)\nabla\Div u_0}{\rho_0}
-u_0\cdot\nabla u_0
-\gamma\rho_0^{\gamma-2}\nabla q_0+\nabla\phi_0
\right]\right|_{\partial\Omega}=0.
\end{equation}
There exist constants $\varepsilon_*>0$ and $C>0$ such that, if
\begin{equation}\label{eq:smallness}
\varepsilon_0
:=\|q_0\|_{H^3}+\|u_0\|_{H^3}+\|\nabla\phi_0\|
\leq\varepsilon_*,
\end{equation}
then the problem \eqref{eq:NSP}, \eqref{eq:data}, with
$\phi(t)\in\dot H^1(\Omega)$, admits a unique global strong solution $(\rho,u,\phi)$. For every $T>0$, this solution satisfies
\begin{equation}\label{eq:regularity}
\begin{aligned}
&q\in C([0,T];H^3),\qquad
  u\in C([0,T];H^3\cap H^1_0)\cap L^2(0,T;H^4),\\
&\nabla\phi\in C([0,T];H^3),\qquad
  q_t\in C([0,T];H^2),\\
&u_t\in C([0,T];H^1_0)\cap L^2(0,T;H^2),\qquad
  u_{tt}\in L^2(0,T;L^2),
\end{aligned}
\end{equation}
and
\begin{equation}\label{eq:density-bounds}
\frac12\leq\rho(t,x)\leq\frac32,
\qquad (t,x)\in[0,\infty)\times\Omega.
\end{equation}
Moreover,
\begin{equation}\label{eq:global-estimate}
\begin{aligned}
&\sup_{t\geq0}\bigl(
\|q(t)\|_{H^3}^2+\|u(t)\|_{H^3}^2
+\|\nabla\phi(t)\|_{H^3}^2
+\|q_t(t)\|_{H^2}^2+\|u_t(t)\|_{H^1}^2\bigr)\\
&\quad+\int_0^\infty\bigl(
\|q\|_{H^3}^2+\|\nabla u\|_{H^3}^2
+\|\nabla u_t\|_{H^1}^2+\|q_t\|_{H^2}^2\bigr)\dd t
\leq C\varepsilon_0^2.
\end{aligned}
\end{equation}
\end{theorem}

To describe the large-time behavior, we impose an additional condition
on the tangential low frequencies of the initial data. Let
$\Lambda_y=(-\Delta_y)^{1/2}$, with fractional powers defined by
\[
\widehat{\Lambda_y^r f}(\xi',x_3)
=|\xi'|^r\widehat f(\xi',x_3),
\]
where the Fourier transform is taken only in $y\in\R^2$.
The following result gives algebraic decay of the solution obtained in
Theorem~\ref{thm:global}.

\begin{theorem}\label{thm:decay}
Under the assumptions of Theorem~\ref{thm:global}, suppose in addition
that, for some $\frac12\leq s<1$,
\begin{equation}\label{eq:negative-data}
M_{s,0}:=\|\Lambda_y^{-s}(q_0,u_0,\nabla\phi_0)\|<\infty.
\end{equation}
Then the global solution satisfies
\begin{equation}\label{eq:solution-decay}
\begin{aligned}
&\|q(t)\|_{H^3}+\|u(t)\|_{H^3}
+\|\nabla\phi(t)\|_{H^3}\\
&\qquad+\|q_t(t)\|_{H^2}+\|u_t(t)\|_{H^1}
\leq C_s(1+t)^{-s/2},\qquad t\geq0,
\end{aligned}
\end{equation}
and
\begin{equation}\label{eq:dissipation-tail}
\int_t^\infty
\bigl(\|q\|_{H^3}^2+\|\nabla u\|_{H^3}^2
+\|\nabla u_t\|_{H^1}^2+\|q_t\|_{H^2}^2\bigr)(\tau)\dd\tau
\leq C_s(1+t)^{-s}.
\end{equation}
The constant $C_s$ may depend on $\mu,\lambda,\gamma,s,\varepsilon_0$
and $M_{s,0}$, but is independent of $t$. No smallness of $M_{s,0}$
is required.
\end{theorem}

We next state the corresponding results for an exterior domain. The
boundary is assumed to be $C^\infty$, and the exterior is connected.
All norms in the following two theorems are taken over this domain.

\begin{theorem}\label{thm:ext-global}
Let $\Omega=\R^3\setminus\overline{\mathcal O}$ be connected, where
$\mathcal O$ is a bounded open set with $C^\infty$ boundary.
Assume \eqref{eq:viscosity} and $\gamma\geq1$. Let
$q_0\in H^3(\Omega)$, $u_0\in H^3(\Omega)\cap H^1_0(\Omega)$,
and let $\phi_0\in\dot H^1(\Omega)$ solve
\eqref{eq:initial-poisson}. Suppose that
\eqref{eq:compatibility} holds. There exist constants
$\varepsilon_*>0$ and $C>0$, depending on $\Omega$ and the physical
parameters, such that, if
\[
\varepsilon_0:=\|q_0\|_{H^3}+\|u_0\|_{H^3}
+\|\nabla\phi_0\|\leq\varepsilon_*,
\]
then \eqref{eq:NSP}, \eqref{eq:data}, with
$\phi(t)\in\dot H^1(\Omega)$, admits a unique global strong solution.
For every $T>0$, this solution has the regularity
\eqref{eq:regularity} and satisfies the pointwise bounds
\eqref{eq:density-bounds}. Moreover,
\begin{equation}\label{eq:ext-global-estimate}
\begin{aligned}
&\sup_{t\geq0}\bigl(
\|q(t)\|_{H^3}^2+\|u(t)\|_{H^3}^2
+\|\nabla\phi(t)\|_{H^3}^2
+\|q_t(t)\|_{H^2}^2+\|u_t(t)\|_{H^1}^2\bigr)\\
&\quad+\int_0^\infty\bigl(
\|q\|_{H^3}^2+\|\nabla u\|_{H^3}^2
+\|\nabla u_t\|_{H^1}^2+\|q_t\|_{H^2}^2\bigr)\dd t
\leq C\varepsilon_0^2.
\end{aligned}
\end{equation}
\end{theorem}

\begin{theorem}\label{thm:ext-decay}
Under the assumptions of Theorem~\ref{thm:ext-global}, suppose in
addition that $q_0,u_0\in L^1(\Omega)$. Then, for every
$0<s<\frac14$, the global solution satisfies
\begin{equation}\label{eq:ext-solution-decay}
\begin{aligned}
&\|q(t)\|_{H^3}+\|u(t)\|_{H^3}
+\|\nabla\phi(t)\|_{H^3}\\
&\qquad+\|q_t(t)\|_{H^2}+\|u_t(t)\|_{H^1}
\leq C_s(1+t)^{-s/2},\qquad t\geq0,
\end{aligned}
\end{equation}
and
\begin{equation}\label{eq:ext-dissipation-tail}
\int_t^\infty
\bigl(\|q\|_{H^3}^2+\|\nabla u\|_{H^3}^2
+\|\nabla u_t\|_{H^1}^2+\|q_t\|_{H^2}^2\bigr)(\tau)\dd\tau
\leq C_s(1+t)^{-s}.
\end{equation}
The constant $C_s$ may depend on
$\Omega,\mu,\lambda,\gamma,s,\varepsilon_0,
\|q_0\|_{L^1}$ and $\|u_0\|_{L^1}$, but is independent of $t$.
No smallness of the two $L^1$ norms is required.
\end{theorem}

\begin{remark}
The smallness condition in Theorem~\ref{thm:global} includes the $L^2$
norm of the initial electric field. This low-order norm is retained in
the energy, while higher derivatives of the potential are controlled
by the density through the Poisson equation. The compatibility condition
\eqref{eq:compatibility} means that $u_t(0)$ has zero trace.
The additional assumption \eqref{eq:negative-data} controls tangential
low frequencies. By the two-dimensional Hardy--Littlewood--Sobolev
inequality, a sufficient condition for \eqref{eq:negative-data} is
\[
q_0,\ u_0,\ \nabla\phi_0
\in L^2\bigl(\R_+;L^{2/(1+s)}(\R^2)\bigr).
\]
\end{remark}

\begin{remark}
The decay exponents in Theorems~\ref{thm:decay} and
\ref{thm:ext-decay} follow from the negative-order estimates used here.
For the half-space, the assumptions control tangential low frequencies;
they differ from the isotropic assumptions in whole-space decay results
such as \cite{LiMatsumuraZhang2010,Wang2012}. For exterior domains, the
restriction $s<\frac14$ arises when a compactly supported error is tested
against $\Lambda^{-2s}f$: the Sobolev order $1+2s$ must be less than
$\frac32$. Establishing optimal rates for either boundary problem
requires further low-frequency information and matching lower bounds.
\end{remark}

In Section~\ref{sec:preliminaries}, we collect the analytic and geometric
lemmas for both domains and state the local existence result.
Sections~\ref{sec:estimates} and~\ref{sec:exterior} establish the uniform
a priori estimates in the half-space and exterior domains, respectively.
The global existence and decay theorems are proved in
Section~\ref{sec:global}.

\section{Preliminaries}\label{sec:preliminaries}

The outward unit normal is denoted by $n$; for $\Omega=\R^3_+$, $n=-e_3$.
In the half-space, we use
$\nabla_y=(\partial_1,\partial_2)$,
$\Delta_y=\partial_1^2+\partial_2^2$ and $u_y=(u_1,u_2)$.
We write $\Div_yu_y=\partial_1u_1+\partial_2u_2$ and use
$\mathbb N_0=\{0,1,2,\ldots\}$.
For a multi-index $\alpha=(\alpha_1,\alpha_2,\alpha_3)\in\mathbb N_0^3$,
set
\[
|\alpha|=\alpha_1+\alpha_2+\alpha_3,\qquad
\partial^\alpha
=\partial_1^{\alpha_1}\partial_2^{\alpha_2}\partial_3^{\alpha_3}.
\]
For a tangential multi-index $\beta=(\beta_1,\beta_2)\in\mathbb N_0^2$, set
$\partial_y^\beta=\partial_1^{\beta_1}\partial_2^{\beta_2}$.
For $k\geq0$, we use the convention
\[
\|\partial_y^k f\|^2
=\sum_{|\beta|=k}\|\partial_y^\beta f\|^2.
\]
Expressions involving $\partial_y^k$ are understood componentwise and
summed over these multi-indices. In particular, $\partial_y^0 f=f$ and
$\|\partial_y^1 f\|=\|\nabla_y f\|$.
The notation $\nabla^k f$ denotes the full array of ordered
spatial derivatives of order $k$.

We use the Lam\'e operator and its bilinear form
\begin{equation}\label{eq:lame-definition}
\begin{aligned}
&Az=-\mu\Delta z-(\mu+\lambda)\nabla\Div z,\\
&a(z,v)=\mu(\nabla z,\nabla v)
       +(\mu+\lambda)(\Div z,\Div v).
\end{aligned}
\end{equation}
By \eqref{eq:viscosity}, $a(z,z)\geq c\|\nabla z\|^2$.
For $v\in H^1_0$, integration by parts gives $(Az,v)=a(z,v)$.

Let $\PP$ be the orthogonal Helmholtz projection in $L^2(\Omega)$
onto divergence-free vector fields with zero normal trace, and put
$\QQ=I-\PP$. The gradient field $\QQ z=\nabla\psi$ is characterized by
\begin{equation}\label{eq:projection-definition}
(\nabla\psi,\nabla\zeta)=(z,\nabla\zeta),
\qquad \zeta\in\dot H^1(\Omega).
\end{equation}
Both projections are bounded in $L^2$. In the half-space they commute
with tangential derivatives.

\subsection{Sobolev and elliptic estimates}

On either domain, we use the Sobolev inequalities
\begin{equation}\label{eq:sobolev}
\|f\|_{L^6}\leq C\|\nabla f\|,\qquad
\|f\|_{L^\infty}\leq C\|f\|_{H^2},\qquad
\|\nabla f\|_{L^\infty}\leq C\|f\|_{H^3},
\end{equation}
for functions in the indicated spaces. In the first inequality,
$f\in\dot H^1(\Omega)$. We also use the product estimate
\begin{equation}\label{eq:sobolev-product}
\|fg\|_{H^m}
\leq C\bigl(\|f\|_{L^\infty}\|g\|_{H^m}
+\|g\|_{L^\infty}\|f\|_{H^m}\bigr),\qquad m=2,3,
\end{equation}
and the corresponding commutator estimate
\begin{equation}\label{eq:sobolev-commutator}
\|[\partial^\alpha,f]\nabla g\|
\leq C\bigl(\|\nabla f\|_{L^\infty}\|\nabla^m g\|
+\|\nabla^m f\|\|\nabla g\|_{L^\infty}\bigr),
\qquad |\alpha|=m,\quad 1\leq m\leq3.
\end{equation}
Here $[\partial^\alpha,f]\nabla g
=\partial^\alpha(f\nabla g)-f\partial^\alpha\nabla g$.

For a fixed bounded set $K=\Omega\cap B_R$ in an exterior domain,
H\"older's inequality and \eqref{eq:sobolev} give
\begin{equation}\label{eq:ext-local-sobolev}
\|f\|_{L^2(K)}\leq |K|^{1/3}\|f\|_{L^6(\Omega)}
\leq C_K\|\nabla f\|_{L^2(\Omega)}.
\end{equation}
The exterior Sobolev inequalities can also be found in
\cite[Lemma 2.1]{DuLiu2025}.

These estimates follow by extension to $\R^3$ and the usual
Sobolev and product inequalities. The elliptic estimates needed below
are recorded next.

\begin{lemma}\label{lem:poisson}
Let $\Omega=\R^3_+$. If $\phi\in\dot H^1(\Omega)$ solves
$\Delta\phi=q$ and $\partial_3\phi=0$ on $\partial\Omega$, then
\begin{equation}\label{eq:poisson-estimate}
\|\nabla^{m+2}\phi\|\leq C\|\nabla^m q\|,
\qquad m=0,1,2.
\end{equation}
For a time-dependent solution, we also have
\begin{equation}\label{eq:poisson-time}
\|\nabla\phi_t\|_{L^6}\leq C\|q_t\|.
\end{equation}
\end{lemma}

\begin{proof}
Extend $\phi$ evenly across $x_3=0$. Since
$\partial_3\phi=0$ there, the reflected equation has no boundary
source. The $L^2$ boundedness of the Fourier multipliers
$\xi_i\xi_j/|\xi|^2$ yields
$\|\nabla^2\phi\|\leq C\|q\|$.
Tangential derivatives satisfy the same Neumann condition, so
the same argument applies to $\partial_y^\beta\phi$.
The remaining normal derivatives follow from
\[
\partial_3^2\phi=q-\Delta_y\phi.
\]
This proves \eqref{eq:poisson-estimate} without any boundary
condition on $q$. Applying the case $m=0$ to $\phi_t$ and using
the homogeneous Sobolev inequality gives
\[
\|\nabla\phi_t\|_{L^6}
\leq C\|\nabla^2\phi_t\|
\leq C\|q_t\|.
\]
\end{proof}

\begin{lemma}\label{lem:ext-poisson}
Let $\Omega$ be a smooth connected exterior domain. Suppose that $\phi\in\dot H^1(\Omega)$ satisfies
$\Delta\phi=q$ and $\partial_n\phi=0$. Then
\begin{equation}\label{eq:ext-poisson}
\|\nabla^{k+2}\phi\|\leq C\|q\|_{H^k},\qquad k=0,1,2.
\end{equation}
Moreover,
\begin{equation}\label{eq:ext-local-field}
\|\nabla\phi\|_{L^6}+\|\nabla\phi\|_{L^2(K)}
\leq C_K\|q\|.
\end{equation}
For a time-dependent solution,
\begin{equation}\label{eq:ext-poisson-time}
\|\nabla\phi_t\|_{L^6}\leq C\|q_t\|.
\end{equation}
\end{lemma}

The refined exterior estimate on the exterior of a ball is also given
in \cite[Lemma 2.8]{DuLiu2025}. We give an argument for the
exterior domains considered here.

\begin{proof}
The usual Neumann estimate first gives
$\nabla^2\phi\in L^2$ when $q\in L^2$ and $\nabla\phi\in L^2$.
The divergence--curl identity for the gradient field, together with
$\partial_n\phi=0$, bounds its boundary contribution by a fixed
multiple of $\|\nabla\phi\|_{L^2(\partial\Omega)}^2$.
The local trace inequality, with a small coefficient on
$\|\nabla^2\phi\|^2$, therefore yields
\begin{equation}\label{eq:ext-neumann-compact}
\|\nabla^2\phi\|
\leq C\bigl(\|q\|+\|\nabla\phi\|_{L^2(K)}\bigr),
\end{equation}
where $K$ is a fixed connected bounded truncation containing the
boundary.

We remove the last term by contradiction. If the estimate with
right-hand side $C\|q\|$ fails, choose $\phi_j$ such that
\[
\|\nabla^2\phi_j\|=1,\qquad \|\Delta\phi_j\|\longrightarrow0.
\]
By \eqref{eq:ext-local-sobolev}, $\nabla\phi_j$ is bounded in $L^6$.
It is bounded in $H^1$ on every bounded truncation, including its
physical boundary. After passage to a subsequence, it converges
strongly in local $L^2$ and weakly in local $H^1$.
Subtracting the means of the potentials on the same connected
truncation, and then enlarging the truncation, shows that the limit
is the gradient of a globally defined function $h$. The weak
Neumann equation passes to the limit, so
\[
\Delta h=0,\qquad \partial_nh=0,\qquad \nabla h\in L^6(\Omega).
\]
The harmonic expansion outside a ball excludes all nonconstant
growing terms, since their gradients do not belong to $L^6$.
Consequently, $h=c+O(|x|^{-1})$ and
$\nabla h=O(|x|^{-2})$ at infinity. Thus
$h-c\in\dot H^1(\Omega)$, and testing the homogeneous Neumann
equation by $h-c$ gives $\nabla h=0$.
The local strong convergence and \eqref{eq:ext-neumann-compact}
now contradict $\|\nabla^2\phi_j\|=1$.
This proves \eqref{eq:ext-poisson} for $k=0$.
Equation \eqref{eq:ext-local-field} follows from
\eqref{eq:ext-local-sobolev}.

For higher derivatives, use local Neumann regularity near the
boundary. On each fixed connected neighborhood, subtract the mean
of $\phi$; its local $L^2$ norm is then bounded by
\eqref{eq:ext-local-field} and Poincar\'e's inequality.
Away from the obstacle, apply the whole-space homogeneous elliptic
estimate to a cutoff of this mean-subtracted potential.
All derivatives of the cutoff have fixed compact support and are
controlled by the estimates already obtained. This gives
\eqref{eq:ext-poisson} for $k=1,2$. Applying the case $k=0$ to
$\phi_t$ proves \eqref{eq:ext-poisson-time}.
\end{proof}

\begin{lemma}\label{lem:elliptic}
In the half-space, for $z=0$ on $\partial\Omega$, the Lam\'e operator satisfies
\begin{equation}\label{eq:lame-estimate}
\|\nabla^2z\|\leq C\|Az\|,
\qquad
\|\nabla^3z\|\leq C\|\nabla Az\|.
\end{equation}
If $(z,\pi)$ solves
\begin{equation}\label{eq:stokes-problem}
-\mu\Delta z+\nabla\pi=g,\qquad
\Div z=h,\qquad z|_{\partial\Omega}=0,
\end{equation}
then
\begin{equation}\label{eq:stokes-estimate}
\|\nabla^{m+2}z\|+\|\nabla^{m+1}\pi\|
\leq C\bigl(\|\nabla^m g\|+\|\nabla^{m+1}h\|\bigr),
\qquad m=0,1,2.
\end{equation}
In a smooth exterior domain, the Dirichlet Lam\'e estimates are
\begin{equation}\label{eq:ext-lame}
\|\nabla^2z\|\leq C(\|Az\|+\|\nabla z\|),
\end{equation}
and
\begin{equation}\label{eq:ext-lame-energy}
\|z\|_{H^3}\leq C(\|Az\|_{H^1}+\|z\|).
\end{equation}
The half-space estimates are understood in the corresponding homogeneous
Sobolev spaces, and $\pi$ is determined up to a constant.
\end{lemma}

\begin{proof}
Since the tangential derivatives of $z$ vanish on the boundary,
\[
(Az,-\Delta_yz)
=\sum_{i=1}^2a(\partial_i z,\partial_i z).
\]
It follows that $\|\nabla\nabla_y z\|\leq C\|Az\|$.
The component equations for $Az$ then give the estimate for
$\partial_3^2z$, since their normal coefficients are $\mu$ and
$2\mu+\lambda$. Applying this argument to the tangential
derivatives of $z$, and differentiating the component equations
once in $x_3$, proves the second estimate in
\eqref{eq:lame-estimate}.

The estimate \eqref{eq:stokes-estimate} is the homogeneous
Dirichlet estimate for the Stokes system in the half-space; see
\cite[Chapter VI]{Galdi2011}. It follows by taking the
Fourier transform in the tangential variables and estimating the
decaying solutions of the resulting ordinary differential system.
Tangential differentiation preserves the zero boundary value.
The remaining normal derivatives are recovered from the component
equations and the divergence equation. This gives the estimates for
$m=0,1,2$ in the homogeneous norms stated above.
For \eqref{eq:ext-lame}, apply local Dirichlet elliptic regularity near
the boundary and the whole-space estimate after a cutoff away from it.
The local zero-order terms are bounded by \eqref{eq:ext-local-sobolev};
the commutators contain at most first derivatives of $z$.
The corresponding estimate with one more derivative gives
\eqref{eq:ext-lame-energy}.
\end{proof}

\begin{lemma}\label{lem:ext-stokes}
For the Stokes problem \eqref{eq:stokes-problem} in the exterior
domain, a fixed connected bounded neighborhood $K$ of the boundary
and the cutoff regions can be chosen so that
\begin{equation}\label{eq:ext-stokes-localized}
\begin{aligned}
\|\nabla^4z\|+\|\nabla^3\pi\|
\leq C\bigl(&\|\nabla^2g\|+\|\nabla^3h\|
+\|g\|_{H^1(K)}+\|h\|_{H^2(K)}\\
&+\|\nabla z\|_{H^2}+\|\nabla\pi\|_{H^1(K)}\bigr).
\end{aligned}
\end{equation}
\end{lemma}

\begin{proof}
Near the boundary use the local Dirichlet--Stokes estimates on
nested smooth neighborhoods; see \cite[Chapters IV--V]{Galdi2011}.
Their highest data norms are
$\|\nabla^2g\|$ and $\|\nabla^3h\|$; the remaining norms are
local. The zero-order velocity norm is bounded by
\eqref{eq:ext-local-sobolev}. Subtracting the mean of $\pi$ on
a fixed connected neighborhood bounds its local $H^2$ norm by
$\|\nabla\pi\|_{H^1(K)}$.

For the far field choose a cutoff $\chi$ which vanishes near the
obstacle and is one outside a larger ball. If $c$ is the mean of
$\pi$ on the fixed neighborhood, the cutoff equations are
\[
\begin{aligned}
-\mu\Delta(\chi z)+\nabla[\chi(\pi-c)]
&=\chi g-2\mu\nabla\chi\cdot\nabla z
-\mu(\Delta\chi)z+(\pi-c)\nabla\chi,\\
\Div(\chi z)&=\chi h+\nabla\chi\cdot z.
\end{aligned}
\]
The whole-space Fourier estimate controls fourth velocity and
third pressure derivatives by two derivatives of the first
right-hand side and three derivatives of the divergence data.
All cutoff terms have fixed compact support and are bounded by
the local norms displayed in \eqref{eq:ext-stokes-localized}.
Only pressure derivatives are estimated; no global $L^2$ norm of
$\pi$ or of $\pi-c$ is required.

\end{proof}

We shall also need estimates involving negative derivatives in the
tangential variables. The restriction to these variables preserves the
boundary conditions.

\begin{lemma}\label{lem:tangential-negative}
Let $\Omega=\R^3_+$ and $\frac12\leq s<1$. For $f\in H^1(\Omega)$ and $g\in L^2(\Omega)$,
\begin{equation}\label{eq:negative-product}
\|\Lambda_y^{-s}(fg)\|
\leq C_s\|f\|^{s-\frac12}
\|\nabla f\|^{\frac32-s}\|g\|.
\end{equation}
Moreover, whenever the norms on the right-hand side are finite,
\begin{equation}\label{eq:tangential-interpolation}
\|f\|\leq
\|\Lambda_y^{-s}f\|^{1/(1+s)}
\|\nabla_y f\|^{s/(1+s)}.
\end{equation}
\end{lemma}

\begin{proof}
For smooth $f$ of compact support in $\overline\Omega$, the tangential
Fourier transform gives, for every $x_3\geq0$,
\[
\begin{aligned}
|\xi'|\,|\widehat f(\xi',x_3)|^2
&=-\int_{x_3}^\infty|\xi'|\,
\partial_z|\widehat f(\xi',z)|^2\dd z\\
&\leq\int_{x_3}^\infty
\bigl(|\xi'|^2|\widehat f(\xi',z)|^2
+|\partial_z\widehat f(\xi',z)|^2\bigr)\dd z.
\end{aligned}
\]
After integration in $\xi'$, the two-dimensional homogeneous Sobolev
inequality yields
\[
\sup_{x_3>0}\|f(\cdot,x_3)\|_{L_y^4}
\leq C\|\nabla f\|.
\]
Similarly,
\[
\sup_{x_3>0}\|f(\cdot,x_3)\|_{L_y^2}
\leq C\|f\|^{1/2}\|\partial_3f\|^{1/2}.
\]
Interpolation between these two bounds gives
\[
\|f\|_{L_{x_3}^\infty L_y^{2/s}}
\leq C_s\|f\|^{s-\frac12}\|\nabla f\|^{\frac32-s}.
\]
Applying the two-dimensional Hardy--Littlewood--Sobolev inequality on
each horizontal slice, followed by H\"older's inequality, we obtain
\[
\|\Lambda_y^{-s}(fg)\|
\leq C_s\|fg\|_{L_{x_3}^2L_y^{2/(1+s)}}
\leq C_s\|f\|_{L_{x_3}^\infty L_y^{2/s}}\|g\|.
\]
This proves \eqref{eq:negative-product}; approximation gives the result
for the stated classes of functions. Finally,
\eqref{eq:tangential-interpolation} follows from H\"older's inequality
in the tangential Fourier variables and $x_3$.
\end{proof}

\subsection{Boundary normal coordinates in exterior domains}

\begin{lemma}\label{lem:normal-coordinates}
A collar of $\partial\Omega$ has a finite covering by coordinates
\begin{equation}\label{eq:normal-chart}
x=X_\ell(y,s)=r_\ell(y)+s\nu_\ell(y),\qquad
0\leq s<s_0,
\end{equation}
where $r_\ell$ parametrizes the boundary and $\nu_\ell=-n$ points
into $\Omega$. Put $e_\alpha=\partial_{y_\alpha}X_\ell$,
$g_{\alpha\beta}=e_\alpha\cdot e_\beta$, and
$t^\alpha=g^{\alpha\beta}e_\beta$, where
$(g^{\alpha\beta})=(g_{\alpha\beta})^{-1}$.
Repeated Greek indices are summed over $1,2$. In each chart,
\begin{equation}\label{eq:coordinate-differential}
\begin{aligned}
&\partial_s=\nu\cdot\nabla,\qquad \partial_s\nu=0,
\qquad \nabla=\nu\partial_s+t^\alpha\partial_{y_\alpha},\\
&\Div v=\nu\cdot\partial_sv+t^\alpha\cdot\partial_{y_\alpha}v,\\
&\Delta v=\partial_s^2v+g^{\alpha\beta}
\partial_{y_\alpha y_\beta}v
+(\Div\nu)\partial_sv+(\Div t^\alpha)\partial_{y_\alpha}v.
\end{aligned}
\end{equation}
Vector fields in these identities have physical Cartesian components.
The volume element is $\dd x=J_\ell(y,s)\dd y\dd s$, where
$J_\ell=\sqrt{\det(g_{\alpha\beta})}$ is bounded above and below
by positive constants on the chosen patches.

There are smooth cutoffs $\chi_\ell$ supported in these charts and
$\chi_0$ vanishing near $\partial\Omega$ and equal to one outside
a ball, such that
\begin{equation}\label{eq:coordinate-cutoffs}
\chi_0^2+\sum_\ell\chi_\ell^2\geq c>0,
\qquad \partial_s\chi_\ell(y,0)=0.
\end{equation}
For smooth $v$ and $f$, the finite family
\begin{equation}\label{eq:localized-operators}
L=\chi_\ell\partial_y^\beta\quad(1\leq|\beta|\leq3),
\qquad L=\chi_0\partial^\alpha\quad(1\leq|\alpha|\leq3),
\end{equation}
has $Lv=0$ on $\partial\Omega$ if $v=0$, and
$\partial_n(Lf)=0$ if $\partial_nf=0$.
\end{lemma}

\begin{proof}
The tubular neighborhood theorem and compactness of the boundary give
\eqref{eq:normal-chart}. Since
$e_\alpha\cdot\nu=0$, $t^\alpha\cdot e_\beta=\delta^\alpha_\beta$
and $\partial_sX_\ell=\nu$, the chain rule gives the gradient and
divergence formulas. Applying the divergence formula to the gradient
gives the Laplacian formula. These identities hold componentwise.

Choose a partition of unity in the surface variables and extend it
constantly along the normal lines near $s=0$. Multiplication by a
cutoff in $s$ that is constant near zero gives the boundary cutoffs.
A further cutoff supported away from the boundary completes the
covering. The cutoffs may be chosen nonnegative, with positive sum,
so their squared sum has a positive lower bound.
For a boundary operator in \eqref{eq:localized-operators},
\[
\partial_s(\chi_\ell\partial_y^\beta f)
=(\partial_s\chi_\ell)\partial_y^\beta f
+\chi_\ell\partial_y^\beta\partial_sf=0\quad(s=0).
\]
The assertion about zero Dirichlet values follows by tangential
differentiation. The operators containing $\chi_0$ vanish near the
boundary.
\end{proof}

The coefficients of the charts and their inverses have bounded
derivatives on each fixed patch. The commutators $[A,L]$, $[L,\nabla]$
and $[L,\Div]$ have orders at most $k+1,k,k$, respectively, when
$L$ has order $k$. The commutator $[L,\Delta]$ has order at most
$k+1$. All these commutators are supported in a fixed bounded set.
For the last commutator, every term contains at least one derivative
of its argument. These facts follow by expanding the product rule
in the Cartesian representation of $L$.

\subsection{Whole-space negative-order estimates}

\begin{lemma}\label{lem:ext-negative}
Let $\Lambda=(-\Delta)^{1/2}$ on $\R^3$ and $0<s<\frac14$.
For $f\in H^1(\R^3)$,
\begin{equation}\label{eq:ext-negative-local}
\|\Lambda^{-2s}f\|_{L^{6/(1-4s)}(\R^3)}
\leq C_s\|\nabla f\|_{L^2(\R^3)}.
\end{equation}
For every fixed bounded measurable set $\mathcal A$,
\begin{equation}\label{eq:ext-negative-annulus}
\|\Lambda^{-2s}f\|_{L^2(\mathcal A)}
\leq C_{s,\mathcal A}\|\nabla f\|_{L^2(\R^3)}.
\end{equation}
In addition,
\begin{align}
\|\nabla\Lambda^{-2s}f\|_{L^2}
&\leq C_s\|f\|_{H^1},\label{eq:ext-negative-density}\\
\|\nabla\Lambda^{-2s}f\|_{L^{6/(3-4s)}}
&\leq C_s\|\nabla f\|_{L^2},\label{eq:ext-negative-gradient}\\
\|f\|_{L^2}
&\leq\|\Lambda^{-s}f\|_{L^2}^{1/(1+s)}
\|\nabla f\|_{L^2}^{s/(1+s)},\label{eq:ext-negative-interpolation}
\end{align}
whenever the right-hand sides are finite.
\end{lemma}

\begin{proof}
The homogeneous Sobolev inequality of order $1+2s<3/2$ gives
\[
\|\Lambda^{-2s}f\|_{L^{6/(1-4s)}}
\leq C_s\|\Lambda^{1+2s}\Lambda^{-2s}f\|_{L^2}
=C_s\|\nabla f\|_{L^2}.
\]
H\"older's inequality on $\mathcal A$ proves
\eqref{eq:ext-negative-annulus}. The bound
$|\xi|^{2-4s}\leq C_s(1+|\xi|^2)$ proves
\eqref{eq:ext-negative-density}. The Hardy--Littlewood--Sobolev
inequality applied to $\Lambda^{-2s}\nabla f$ gives
\eqref{eq:ext-negative-gradient}. Finally, write
\[
|\widehat f|^2
=\bigl(|\xi|^{-2s}|\widehat f|^2\bigr)^{1/(1+s)}
\bigl(|\xi|^2|\widehat f|^2\bigr)^{s/(1+s)}
\]
and apply H\"older's inequality in Fourier space.
\end{proof}

\subsection{Perturbation equations and local solutions}

Dividing the momentum equation by $1+q$, we rewrite
\eqref{eq:NSP} as
\begin{equation}\label{eq:perturbation}
\left\{
\begin{aligned}
&q_t+\Div u=-\Div(qu),\\
&u_t+Au+\gamma\nabla q-\nabla\phi=F,\\
&\Delta\phi=q,
\end{aligned}
\right.
\end{equation}
where
\begin{equation}\label{eq:F}
F=-u\cdot\nabla u+\frac{q}{1+q}Au
-\gamma\bigl((1+q)^{\gamma-2}-1\bigr)\nabla q.
\end{equation}
The continuity and Poisson equations imply
\begin{equation}\label{eq:field-time}
\nabla\phi_t=-\QQ((1+q)u).
\end{equation}
Indeed, in the duality between $\dot H^1(\Omega)$ and its dual,
\[
(\nabla\phi_t,\nabla\zeta)
=-\langle q_t,\zeta\rangle
=-((1+q)u,\nabla\zeta),
\]
for $\zeta\in\dot H^1(\Omega)$. Here the boundary term vanishes since $u=0$ on $\partial\Omega$.
This also gives
\begin{equation}\label{eq:field-time-bound}
\|\nabla\phi_t\|\leq C\|u\|,
\end{equation}
whenever $\|q\|_{L^\infty}\leq1/2$.

The local existence and continuation result is stated next.

\begin{lemma}\label{lem:local}
Let $\Omega$ be either domain described in Section~\ref{sec:introduction}.
Let $q_0\in H^3$, $u_0\in H^3\cap H^1_0$, and let
$\phi_0\in\dot H^1(\Omega)$ satisfy \eqref{eq:initial-poisson}.
Assume \eqref{eq:viscosity}, \eqref{eq:compatibility},
$\gamma\geq1$, and $\inf_\Omega(1+q_0)>0$.
Then there exists $T_0>0$ such that the problem
\eqref{eq:NSP}, \eqref{eq:data}, with
$\phi(t)\in\dot H^1(\Omega)$, has a unique strong solution
with the regularity \eqref{eq:regularity} on $[0,T_0]$.
If its maximal existence time $T^*$ is finite, then either
\[
\liminf_{t\uparrow T^*}\inf_{x\in\Omega}(1+q(t,x))=0
\]
or
\[
\sup_{0\leq t<T^*}
\bigl(\|q(t)\|_{H^3}+\|u(t)\|_{H^3}
+\|\nabla\phi(t)\|\bigr)=\infty.
\]
\end{lemma}

\begin{proof}
We use the local iteration for positive-density compressible
viscous flow with the homogeneous Dirichlet condition, as in
\cite{MatsumuraNishida1983}. We indicate how the Poisson equation
is included in this construction. Along the continuity equation,
the electric field is determined by
\[
\nabla\phi(t)=\nabla\phi_0
-\int_0^t\QQ\bigl((1+q)u\bigr)(s)\dd s.
\]
This identity preserves the Poisson equation and the homogeneous
Neumann condition in their weak form. Moreover,
\[
\|\nabla\phi(t)\|
\leq\|\nabla\phi_0\|
+\int_0^t\|(1+q)u\|(s)\dd s,
\qquad
\|\nabla\phi(t)\|_{H^3}
\leq C\bigl(\|\nabla\phi(t)\|+\|q(t)\|_{H^2}\bigr).
\]
The second estimate follows from Lemmas~\ref{lem:poisson}
and~\ref{lem:ext-poisson} in the respective domains.
For two such triples $(q^{(i)},u^{(i)},\phi^{(i)})$, $i=1,2$,
with the same initial data and a common $H^3$ bound for $q^{(i)}$
and $u^{(i)}$, the $L^2$ boundedness of $\QQ$ gives
\[
\begin{aligned}
\|\nabla(\phi^{(1)}-\phi^{(2)})(t)\|
&\leq\int_0^t
\|(1+q^{(1)})u^{(1)}-(1+q^{(2)})u^{(2)}\|(s)\dd s\\
&\leq C\int_0^t
\bigl(\|u^{(1)}-u^{(2)}\|+\|q^{(1)}-q^{(2)}\|\bigr)(s)\dd s.
\end{aligned}
\]
Here $C$ depends on the common $H^3$ bound. Applying
Lemma~\ref{lem:poisson} to the difference of the Poisson equations
also yields
\[
\|\nabla(\phi^{(1)}-\phi^{(2)})\|_{H^3}
\leq C\bigl(
\|\nabla(\phi^{(1)}-\phi^{(2)})\|
+\|q^{(1)}-q^{(2)}\|_{H^2}\bigr).
\]
These estimates give the required local Lipschitz bound for the electric
force. Thus the transport estimates for $q$ and the Dirichlet parabolic
estimates for $u$
give a uniform bound and convergence on a sufficiently short
time interval. The compatibility condition gives
$u_t(0)\in H^1_0$ and hence the stated time regularity.
The existence time depends only on a positive lower bound for
$1+q_0$ and an upper bound for
$\|q_0\|_{H^3}+\|u_0\|_{H^3}+\|\nabla\phi_0\|$.
Restarting the same construction proves the continuation
criterion.
\end{proof}

\section{Uniform estimates in the half-space}\label{sec:estimates}

In this section $\Omega=\R^3_+$. Let $(q,u,\phi)$ be a solution on $[0,T]$ with the regularity
\eqref{eq:regularity}. Set
\begin{equation}\label{eq:ND}
\begin{aligned}
&N(t)^2=\|q(t)\|_{H^3}^2+\|u(t)\|_{H^3}^2
       +\|\nabla\phi(t)\|^2+\|u_t(t)\|_{H^1}^2,\\
&D(t)=\|q(t)\|_{H^3}^2+\|\nabla u(t)\|_{H^3}^2
       +\|\nabla u_t(t)\|_{H^1}^2.
\end{aligned}
\end{equation}
Assume that
\begin{equation}\label{eq:apriori-assumption}
\sup_{0\leq t\leq T}N(t)\leq\delta,
\end{equation}
where $0<\delta\leq1$ is sufficiently small. In particular,
$\|q\|_{L^\infty}\leq1/2$. The constants below are independent of $T$.

Writing the continuity equation as
\[
q_t=-\Div u-u\cdot\nabla q-q\Div u,
\]
and applying the Sobolev product estimates, we obtain
\begin{equation}\label{eq:qt-estimates}
\begin{aligned}
&\|\nabla q_t\|
\leq C\|\nabla u\|_{H^1}+CN(t)\|q\|_{H^2},\\
&\|q_t\|_{H^2}
\leq C\|\nabla u\|_{H^2}+CN(t)\|q\|_{H^3}.
\end{aligned}
\end{equation}
In particular, $\|q_t\|_{H^2}\leq CN(t)$ and
$\|q_t\|_{H^2}\leq CD(t)^{1/2}$. Since $\|q\|_{L^\infty}\leq1/2$,
the coefficients in \eqref{eq:F} satisfy
\[
\left\|\frac{q}{1+q}\right\|_{H^m}
+\|(1+q)^{\gamma-2}-1\|_{H^m}
\leq C\|q\|_{H^m},\qquad m=2,3.
\]
It follows that
\begin{equation}\label{eq:products}
\|F\|_{H^2}+\|q\Div u\|_{H^3}\leq CN(t)D(t)^{1/2},
\end{equation}
where the nonlinear terms are bounded by
\begin{equation}\label{eq:products-detail}
\begin{aligned}
\|u\cdot\nabla u\|_{H^2}
&\leq C\|u\|_{H^2}\|\nabla u\|_{H^2},\\
\left\|\frac{q}{1+q}Au\right\|_{H^2}
&\leq C\|q\|_{H^2}\|\nabla^2u\|_{H^2},\\
\|((1+q)^{\gamma-2}-1)\nabla q\|_{H^2}
&\leq C\|q\|_{H^2}\|\nabla q\|_{H^2},\\
\|q\Div u\|_{H^3}
&\leq C\|q\|_{H^3}\|\nabla u\|_{H^3}.
\end{aligned}
\end{equation}
Each right-hand side is at most $CN(t)D(t)^{1/2}$.
Differentiating \eqref{eq:F} in time gives
\begin{equation}\label{eq:Ft}
\begin{aligned}
F_t={}&-u_t\cdot\nabla u-u\cdot\nabla u_t
+\frac{q}{1+q}Au_t+\frac{q_t}{(1+q)^2}Au\\
&-\gamma\nabla\bigl[((1+q)^{\gamma-2}-1)q_t\bigr].
\end{aligned}
\end{equation}
Hence
\begin{equation}\label{eq:Ft-bound}
\|F_t\|\leq CN(t)D(t)^{1/2},
\end{equation}
where, by \eqref{eq:qt-estimates},
\begin{equation}\label{eq:Ft-detail}
\begin{aligned}
\|u_t\cdot\nabla u\|
&\leq\|u_t\|_{L^3}\|\nabla u\|_{L^6}
\leq CN(t)D(t)^{1/2},\\
\|u\cdot\nabla u_t\|
&\leq\|u\|_{L^\infty}\|\nabla u_t\|
\leq CN(t)D(t)^{1/2},\\
\left\|\frac{q}{1+q}Au_t\right\|
&\leq C\|q\|_{L^\infty}\|\nabla^2u_t\|
\leq CN(t)D(t)^{1/2},\\
\left\|\frac{q_t}{(1+q)^2}Au\right\|
&\leq C\|q_t\|_{L^\infty}\|\nabla^2u\|
\leq CN(t)D(t)^{1/2},\\
\|\nabla[((1+q)^{\gamma-2}-1)q_t]\|
&\leq C\|q\|_{L^\infty}\|\nabla q_t\|
+C\|q_t\|_{L^\infty}\|\nabla q\|\\
&\leq CN(t)D(t)^{1/2}.
\end{aligned}
\end{equation}
We shall use these bounds in the following estimates.

We begin with the basic energy identity.

\begin{lemma}\label{lem:basic}
Let $G(1)=G'(1)=0$ and $G''(\rho)=p'(\rho)/\rho$. Define
\begin{equation}\label{eq:E0}
\mathcal E_0(t)=\int_\Omega
\left(\frac12\rho|u|^2+G(\rho)+\frac12|\nabla\phi|^2\right)\dd x.
\end{equation}
Then
\begin{equation}\label{eq:basic-energy}
\frac{\dd}{\dd t}\mathcal E_0+a(u,u)=0,
\end{equation}
and $\mathcal E_0$ is equivalent to
$\|q\|^2+\|u\|^2+\|\nabla\phi\|^2$ under
\eqref{eq:apriori-assumption}.
\end{lemma}

\begin{proof}
The definition of $G$ implies
\[
\rho G'(\rho)-G(\rho)=p(\rho)-p(1).
\]
The continuity equation therefore yields
\[
\partial_tG(\rho)+\Div(G(\rho)u)
+(p(\rho)-p(1))\Div u=0.
\]
Taking the inner product of the momentum equation in
\eqref{eq:NSP} with $u$ and adding the integral of this identity,
we obtain
\[
\frac{\dd}{\dd t}\int_\Omega
\left(\frac12\rho|u|^2+G(\rho)\right)\dd x
+a(u,u)=(\rho u,\nabla\phi).
\]
Here the continuity equation gives the kinetic energy derivative,
and the pressure terms cancel. The boundary fluxes vanish since $u=0$. By \eqref{eq:field-time} and the
self-adjointness of $\QQ$,
\[
\frac12\frac{\dd}{\dd t}\|\nabla\phi\|^2
=-(\rho u,\nabla\phi).
\]
Adding these equalities proves \eqref{eq:basic-energy}.
Finally, Taylor's formula gives
\[
G(1+q)=\gamma q^2\int_0^1(1-s)(1+sq)^{\gamma-2}\dd s,
\]
which is bounded above and below by positive multiples of $q^2$
for $|q|\leq1/2$. This proves the asserted equivalence.
\end{proof}

Differentiation in time preserves both boundary conditions. We use this
fact to estimate the first time derivatives.

\begin{lemma}\label{lem:time}
Set
\begin{equation}\label{eq:E1}
\mathcal E_1(t)=\frac12
\bigl(\gamma\|q_t\|^2+\|u_t\|^2+\|\nabla\phi_t\|^2\bigr).
\end{equation}
Then
\begin{equation}\label{eq:time-energy}
\frac{\dd}{\dd t}\mathcal E_1+c\|\nabla u_t\|^2\leq CN(t)D(t).
\end{equation}
\end{lemma}

\begin{proof}
Differentiating the first two equations of \eqref{eq:perturbation}
in time and testing them with $\gamma q_t$ and $u_t$, respectively,
gives
\begin{equation}\label{eq:time-identity}
\frac{\dd}{\dd t}\mathcal E_1+a(u_t,u_t)
=(F_t,u_t)+\gamma((qu)_t,\nabla q_t)
-(\nabla\phi_t,(qu)_t).
\end{equation}
Here $u_t=0$ and $(qu)_t=0$ on the boundary. Differentiating
\eqref{eq:field-time} and using the orthogonality of $\QQ$, we have
\[
\frac12\frac{\dd}{\dd t}\|\nabla\phi_t\|^2
=-(u_t,\nabla\phi_t)-((qu)_t,\nabla\phi_t),
\]
which supplies the electric part of \eqref{eq:time-identity}.

We estimate $(F_t,u_t)$ using \eqref{eq:Ft}. For its principal
viscous term, integration by parts yields
\begin{equation}\label{eq:time-viscous}
\begin{aligned}
\left(\frac{q}{1+q}Au_t,u_t\right)
={}&\int_\Omega\frac{q}{1+q}
\left(\mu|\nabla u_t|^2+(\mu+\lambda)|\Div u_t|^2\right)\dd x\\
&+\mu\sum_{i,j=1}^3\int_\Omega
\partial_j\left(\frac{q}{1+q}\right)
\partial_j u_{t,i}\,u_{t,i}\dd x\\
&+(\mu+\lambda)\int_\Omega
\Div u_t\,\nabla\left(\frac{q}{1+q}\right)\cdot u_t\dd x.
\end{aligned}
\end{equation}
The right-hand side is bounded in absolute value by
\[
C\|q\|_{L^\infty}\|\nabla u_t\|^2
+C\|\nabla q\|_{L^3}\|\nabla u_t\|\|u_t\|_{L^6}
\leq CN(t)\|\nabla u_t\|^2.
\]
We used $\|u_t\|_{L^6}\leq C\|\nabla u_t\|$ in the last step.
The other viscous term satisfies
\[
\left|\left(\frac{q_t}{(1+q)^2}Au,u_t\right)\right|
\leq C\|q_t\|_{L^3}\|Au\|\|u_t\|_{L^6}
\leq CN(t)D(t).
\]
Integration by parts in the pressure term and H\"older's
inequality give
\begin{align*}
\left|\left(\nabla[((1+q)^{\gamma-2}-1)q_t],u_t\right)\right|
&\leq C\|q\|_{L^\infty}\|q_t\|\|\nabla u_t\|
\leq CN(t)D(t),\\
|(u_t\cdot\nabla u,u_t)|
&\leq\|u_t\|_{L^3}\|\nabla u\|\|u_t\|_{L^6}
\leq CN(t)D(t),\\
|(u\cdot\nabla u_t,u_t)|
&=\frac12|(\Div u,|u_t|^2)|\\
&\leq C\|\nabla u\|\|u_t\|_{L^3}\|u_t\|_{L^6}
\leq CN(t)D(t).
\end{align*}
Consequently, $|(F_t,u_t)|\leq CN(t)D(t)$.
For the continuity term, we write $(qu)_t=q_tu+qu_t$ and obtain
\begin{align*}
|(q_tu,\nabla q_t)|
&=\frac12|(\Div u,q_t^2)|
\leq C\|\nabla u\|_{L^\infty}\|q_t\|^2
\leq CN(t)D(t),\\
|(qu_t,\nabla q_t)|
&\leq\|q\|_{L^3}\|u_t\|_{L^6}\|\nabla q_t\|
\leq CN(t)D(t).
\end{align*}
Finally, Lemma~\ref{lem:poisson} gives
\begin{align*}
|(\nabla\phi_t,(qu)_t)|
&\leq\|\nabla\phi_t\|_{L^6}
\bigl(\|q_t\|\|u\|_{L^3}+\|q\|\|u_t\|_{L^3}\bigr)\\
&\leq CN(t)\bigl(\|q_t\|^2+\|q\|^2\bigr)
\leq CN(t)D(t).
\end{align*}
Substituting these bounds into \eqref{eq:time-identity} proves
\eqref{eq:time-energy}.
\end{proof}

The previous identities give the viscous dissipation. The following
estimate supplies the $L^2$ and $H^1$ control of the density in the
time integral.

\begin{lemma}\label{lem:density}
The density satisfies
\begin{equation}\label{eq:density-energy}
\frac{\dd}{\dd t}(u,\nabla q)+c\|q\|_{H^1}^2
\leq C\|\nabla u\|_{H^1}^2+CN(t)D(t).
\end{equation}
\end{lemma}

\begin{proof}
Taking the inner product of the momentum equation in
\eqref{eq:perturbation} with $\nabla q$ and using the continuity
equation, we find
\begin{equation}\label{eq:density-identity}
\begin{aligned}
\frac{\dd}{\dd t}(u,\nabla q)
+\gamma\|\nabla q\|^2+\|q\|^2
={}&\|\Div u\|^2-(Au,\nabla q)+(F,\nabla q)\\
&+(\Div(qu),\Div u).
\end{aligned}
\end{equation}
In this calculation,
\[
(u,\nabla q_t)=-(\Div u,q_t),\qquad
(\nabla\phi,\nabla q)=-\|q\|^2,
\]
because $u\cdot n=0$ and $\partial_n\phi=0$. Moreover,
\begin{align*}
|(Au,\nabla q)|
&\leq\frac\gamma2\|\nabla q\|^2+C\|\nabla^2u\|^2,\\
|(F,\nabla q)|&\leq\|F\|\|\nabla q\|\leq CN(t)D(t),\\
|(\Div(qu),\Div u)|
&\leq C\bigl(\|u\|_{L^\infty}\|\nabla q\|
+\|q\|_{L^\infty}\|\nabla u\|\bigr)\|\nabla u\|
\leq CN(t)D(t).
\end{align*}
Substitution into \eqref{eq:density-identity} leaves positive
multiples of $\|q\|^2$ and $\|\nabla q\|^2$ on the left-hand
side and proves \eqref{eq:density-energy}.
\end{proof}

We next estimate the spatial gradient of $u_t$. In this estimate,
the electric contribution at the boundary is retained through a
correction to the energy.

\begin{lemma}\label{lem:gradient-time}
Define
\begin{equation}\label{eq:E2}
\mathcal E_2(t)
=\frac12a(u_t,u_t)+(\QQ u,Au)-\frac12a(u,u).
\end{equation}
Then
\begin{equation}\label{eq:gradient-time-energy}
\frac{\dd}{\dd t}\mathcal E_2+c\|\nabla^2u_t\|^2
\leq C\|\nabla u\|_{H^1}^2+CN(t)D(t).
\end{equation}
\end{lemma}

\begin{proof}
Direct integration by parts gives
\begin{equation}\label{eq:electric-boundary}
\begin{aligned}
(\nabla\phi_t,Au_t)
={}&\mu(\nabla^2\phi_t,\nabla u_t)
+(\mu+\lambda)(q_t,\Div u_t)\\
&-\mu\int_{\partial\Omega}
\nabla_y\phi_t\cdot\partial_n u_{t,y}\dd y.
\end{aligned}
\end{equation}
The last integral need not vanish, since the no-slip condition does not
force $\partial_n u_{t,y}$ to vanish on the boundary.
To retain its contribution, we apply $\PP$ to the momentum equation and use
\eqref{eq:field-time}:
\[
\PP u_t=-\PP Au+\PP F,\qquad
\QQ u_t=u_t+\PP Au-\PP F,\qquad
\nabla\phi_t=-\QQ u-\QQ(qu).
\]
It follows that
\begin{equation}\label{eq:projection-energy-identity}
\begin{aligned}
(\nabla\phi_t,Au_t)
={}&-\frac{\dd}{\dd t}(\QQ u,Au)
+(\QQ u_t,Au)-(\QQ(qu),Au_t)\\
={}&\frac{\dd}{\dd t}
\left[-(\QQ u,Au)+\frac12a(u,u)\right]
+\|\PP Au\|^2\\
&-(\PP F,Au)-(\QQ(qu),Au_t).
\end{aligned}
\end{equation}
Here $(u_t,Au)=a(u,u_t)$ since $u_t=0$ on the boundary.

Testing the time derivative of the momentum equation with $Au_t$
and substituting \eqref{eq:projection-energy-identity}, we obtain
\begin{equation}\label{eq:gradient-time-identity}
\begin{aligned}
\frac{\dd}{\dd t}\mathcal E_2+\|Au_t\|^2
={}&\|\PP Au\|^2-(\PP F,Au)-(\QQ(qu),Au_t)\\
&-\gamma(\nabla q_t,Au_t)+(F_t,Au_t).
\end{aligned}
\end{equation}
Since
\[
\|\QQ(qu)\|\leq\|u\|_{L^\infty}\|q\|,
\]
the $L^2$ boundedness of the projections and
\eqref{eq:products}--\eqref{eq:Ft-bound} give
\begin{align*}
\|\PP Au\|^2&\leq C\|\nabla^2u\|^2,\\
|(\PP F,Au)|&\leq\|F\|\|Au\|\leq CN(t)D(t),\\
|(\QQ(qu),Au_t)|+|(F_t,Au_t)|
&\leq\epsilon\|Au_t\|^2+C_\epsilon N(t)^2D(t).
\end{align*}
For the pressure term, \eqref{eq:qt-estimates} implies
\[
\gamma|(\nabla q_t,Au_t)|
\leq\epsilon\|Au_t\|^2
+C_\epsilon\bigl(\|\nabla u\|_{H^1}^2
+N(t)^2\|q\|_{H^2}^2\bigr).
\]
Choosing $\epsilon$ sufficiently small absorbs the two terms
containing $\|Au_t\|^2$. Since $N(t)\leq\delta\leq1$ and
$\|\nabla^2u_t\|\leq C\|Au_t\|$ by
Lemma~\ref{lem:elliptic}, these bounds prove
\eqref{eq:gradient-time-energy}.
\end{proof}

Since tangential differentiation preserves the boundary conditions,
we can now derive the estimates for tangential derivatives up to order
three.

\begin{lemma}\label{lem:tangential}
Under \eqref{eq:apriori-assumption},
\begin{equation}\label{eq:tangential-energy}
\begin{aligned}
&\frac12\frac{\dd}{\dd t}
\sum_{1\leq|\beta|\leq3}
\left(\gamma\|\partial_y^\beta q\|^2
+\|\partial_y^\beta u\|^2
+\|\nabla\partial_y^\beta\phi\|^2\right)\\
&\qquad+c\sum_{1\leq|\beta|\leq3}
\|\nabla\partial_y^\beta u\|^2\leq CN(t)D(t).
\end{aligned}
\end{equation}
\end{lemma}

\begin{proof}
Tangential derivatives preserve the boundary conditions
$\partial_y^\beta u=0$ and
$\partial_n\partial_y^\beta\phi=0$. Applying
$\partial_y^\beta$ to \eqref{eq:perturbation}, taking the
corresponding inner products, and using \eqref{eq:field-time},
we obtain
\begin{equation}\label{eq:tangential-identity}
\begin{aligned}
&\frac12\frac{\dd}{\dd t}
\left(\gamma\|\partial_y^\beta q\|^2
+\|\partial_y^\beta u\|^2
+\|\nabla\partial_y^\beta\phi\|^2\right)
+a(\partial_y^\beta u,\partial_y^\beta u)\\
&\quad=(\partial_y^\beta F,\partial_y^\beta u)
-\gamma(\partial_y^\beta\Div(qu),\partial_y^\beta q)
-(\nabla\partial_y^\beta\phi,\partial_y^\beta(qu)).
\end{aligned}
\end{equation}
For the continuity term, write
\begin{equation}\label{eq:tangential-transport}
\begin{aligned}
\partial_y^\beta\Div(qu)
={}&u\cdot\nabla\partial_y^\beta q
+q\Div\partial_y^\beta u\\
&+[\partial_y^\beta,u\cdot\nabla]q
+[\partial_y^\beta,q]\Div u.
\end{aligned}
\end{equation}
The zero boundary flux and \eqref{eq:sobolev-commutator} give
\begin{align*}
(u\cdot\nabla\partial_y^\beta q,\partial_y^\beta q)
&=-\frac12(\Div u,|\partial_y^\beta q|^2),\\
\|[\partial_y^\beta,u\cdot\nabla]q\|
+\|[\partial_y^\beta,q]\Div u\|
&\leq C\bigl(\|\nabla u\|_{L^\infty}\|q\|_{H^3}
+\|\nabla q\|_{L^\infty}\|\nabla u\|_{H^2}\bigr)
\leq CN(t)D(t)^{1/2}.
\end{align*}
Together with
\[
|(q\Div\partial_y^\beta u,\partial_y^\beta q)|
\leq\|q\|_{L^\infty}\|\Div\partial_y^\beta u\|
\|\partial_y^\beta q\|\leq CN(t)D(t),
\]
these estimates bound the continuity term in
\eqref{eq:tangential-identity} by $CN(t)D(t)$.
The convective term in $F$ is treated by the same integration:
\begin{align*}
|(\partial_y^\beta(u\cdot\nabla u),\partial_y^\beta u)|
&\leq\frac12\|\Div u\|_{L^\infty}\|\partial_y^\beta u\|^2
+\|[\partial_y^\beta,u\cdot\nabla]u\|
\|\partial_y^\beta u\|\\
&\leq CN(t)\|\nabla u\|_{H^2}^2\leq CN(t)D(t).
\end{align*}

For the nonlinear pressure, integration by parts gives
\begin{equation}\label{eq:pressure-primitive}
\begin{aligned}
&\left(-\partial_y^\beta
\left[\gamma\bigl((1+q)^{\gamma-2}-1\bigr)\nabla q\right],
\partial_y^\beta u\right)\\
&\qquad=
\left(\partial_y^\beta
\int_0^q\gamma\bigl((1+s)^{\gamma-2}-1\bigr)\dd s,
\Div\partial_y^\beta u\right).
\end{aligned}
\end{equation}
The primitive vanishes to second order at $q=0$. Hence the
composition and product estimates imply
\[
\left\|\partial_y^\beta
\int_0^q\gamma\bigl((1+s)^{\gamma-2}-1\bigr)\dd s\right\|
\leq C\|q\|_{L^\infty}\|q\|_{H^3}
\leq CN(t)\|q\|_{H^3}.
\]
Multiplication by $\|\Div\partial_y^\beta u\|$ bounds
\eqref{eq:pressure-primitive} by $CN(t)D(t)$.

For the viscous term in $F$, expand
\begin{equation}\label{eq:tangential-viscous}
\partial_y^\beta\left(\frac{q}{1+q}Au\right)
=\frac{q}{1+q}A\partial_y^\beta u
+\sum_{0<\zeta\leq\beta}\binom\beta\zeta
\partial_y^\zeta\left(\frac{q}{1+q}\right)
A\partial_y^{\beta-\zeta}u.
\end{equation}
The first term is integrated once, with zero boundary contribution,
to obtain
\begin{align*}
\left|\left(\frac{q}{1+q}A\partial_y^\beta u,
\partial_y^\beta u\right)\right|
&\leq C\left\|\frac{q}{1+q}\right\|_{L^\infty}
\|\nabla\partial_y^\beta u\|^2\\
&\quad+C\left\|\nabla\left(\frac{q}{1+q}\right)\right\|_{L^3}
\|\nabla\partial_y^\beta u\|\|\partial_y^\beta u\|_{L^6}
\leq CN(t)\|\nabla\partial_y^\beta u\|^2.
\end{align*}
For $m=|\beta|$ and $0<\zeta\leq\beta$, the terms with
$|\zeta|=1,2$ satisfy, respectively,
\begin{align*}
&\left|\left(\partial_y^\zeta\left(\frac{q}{1+q}\right)
A\partial_y^{\beta-\zeta}u,\partial_y^\beta u\right)\right|\\
&\quad\leq
\begin{cases}
C\left\|\nabla\left(\dfrac{q}{1+q}\right)\right\|_{L^\infty}
\|\nabla^{m+1}u\|\|\partial_y^\beta u\|,& |\zeta|=1,\\[1mm]
C\left\|\nabla^2\left(\dfrac{q}{1+q}\right)\right\|_{L^3}
\|\nabla^m u\|_{L^6}\|\partial_y^\beta u\|,& |\zeta|=2,
\end{cases}
\end{align*}
and both bounds are at most $CN(t)D(t)$ by \eqref{eq:sobolev}.
When $|\beta|=3$, the term containing the highest derivative of
the coefficient satisfies
\begin{align*}
&\left|\left(
\partial_y^\beta\left(\frac{q}{1+q}\right)Au,
\partial_y^\beta u\right)\right|\\
&\qquad\leq
\left\|\partial_y^\beta\left(\frac{q}{1+q}\right)\right\|
\|Au\|_{L^3}\|\partial_y^\beta u\|_{L^6}
\leq CN(t)D(t).
\end{align*}
Here the coefficient norms are bounded by $CN(t)$, while
$\|Au\|_{L^3}\leq C\|\nabla u\|_{H^2}$ and
$\|\partial_y^\beta u\|_{L^6}
\leq C\|\nabla\partial_y^\beta u\|$. 

Finally, Lemma~\ref{lem:poisson} implies
\[
\|\nabla\partial_y^\beta\phi\|
\leq C\|q\|_{H^{|\beta|-1}},\qquad 1\leq|\beta|\leq3.
\]
Using the product estimate for derivatives of positive order,
\[
\|\partial_y^\beta(qu)\|
\leq C\bigl(\|u\|_{L^\infty}\|q\|_{H^3}
+\|q\|_{L^\infty}\|\nabla u\|_{H^2}\bigr)
\leq CN(t)D(t)^{1/2},
\]
we also have
$|(\nabla\partial_y^\beta\phi,\partial_y^\beta(qu))|\leq CN(t)D(t)$.
Substituting these bounds into \eqref{eq:tangential-identity}
and summing over $\beta$ proves \eqref{eq:tangential-energy}.
\end{proof}

To estimate normal derivatives of $q$, it is convenient to retain
its material derivative
\begin{equation}\label{eq:w}
w=q_t+u\cdot\nabla q=-(1+q)\Div u.
\end{equation}
This form avoids estimating a fourth derivative of $q$ in the
transport term. The normal component of the momentum equation gives
\begin{equation}\label{eq:normal-momentum}
\begin{aligned}
(2\mu+\lambda)\partial_3w+\gamma\partial_3q
={}&\partial_3\phi-\partial_tu_3
+\mu(\Delta_yu_3-\partial_3\Div_yu_y)\\
&-(2\mu+\lambda)\partial_3(q\Div u)+F_3.
\end{aligned}
\end{equation}

\begin{lemma}\label{lem:normal}
There exist fixed constants $0<\eta<1$ and $M_4>0$ such that
\begin{equation}\label{eq:H}
\begin{aligned}
\mathcal H(t)={}&\frac{2\mu+\lambda+\gamma}{2}
\sum_{j=0}^2\eta^j
\|\partial_y^{2-j}\partial_3^{j+1}q\|^2\\
&+\frac{M_4}{2}\sum_{1\leq|\beta|\leq3}
\left(\gamma\|\partial_y^\beta q\|^2
+\|\partial_y^\beta u\|^2
+\|\nabla\partial_y^\beta\phi\|^2\right)
\end{aligned}
\end{equation}
satisfies
\begin{equation}\label{eq:spatial-energy}
\begin{aligned}
\frac{\dd}{\dd t}\mathcal H
+c\bigl(\|q\|_{H^3}^2+\|\nabla u\|_{H^3}^2\bigr)
\leq C\bigl(\|\nabla^2u_t\|^2+\|q\|^2
+\|\nabla u\|^2\bigr)+CN(t)D(t).
\end{aligned}
\end{equation}
\end{lemma}

\begin{proof}
For $j=0,1,2$, apply $\partial_y^{2-j}\partial_3^j$ to
\eqref{eq:normal-momentum}. Take the inner product with
$\partial_y^{2-j}\partial_3^{j+1}(q+w)$ and sum over the
tangential multi-indices. The left-hand side of the resulting
identity is
\begin{align*}
&(2\mu+\lambda)
\|\partial_y^{2-j}\partial_3^{j+1}w\|^2
+\gamma\|\partial_y^{2-j}\partial_3^{j+1}q\|^2\\
&\qquad+(2\mu+\lambda+\gamma)
(\partial_y^{2-j}\partial_3^{j+1}w,
\partial_y^{2-j}\partial_3^{j+1}q).
\end{align*}
By \eqref{eq:w},
\begin{equation}\label{eq:normal-transport}
\begin{aligned}
\partial_y^{2-j}\partial_3^{j+1}w
={}&\partial_t\partial_y^{2-j}\partial_3^{j+1}q
+u\cdot\nabla\partial_y^{2-j}\partial_3^{j+1}q\\
&+[\partial_y^{2-j}\partial_3^{j+1},u\cdot\nabla]q.
\end{aligned}
\end{equation}
The commutator is bounded by
\[
\|[\partial_y^{2-j}\partial_3^{j+1},u\cdot\nabla]q\|
\leq C\bigl(\|\nabla u\|_{L^\infty}\|\nabla^3q\|
+\|\nabla^3u\|\|\nabla q\|_{L^\infty}\bigr)
\leq CN(t)D(t)^{1/2}.
\]
Consequently,
\begin{align*}
&(\partial_y^{2-j}\partial_3^{j+1}w,
\partial_y^{2-j}\partial_3^{j+1}q)\\
&\quad=\frac12\frac{\dd}{\dd t}
\|\partial_y^{2-j}\partial_3^{j+1}q\|^2
-\frac12(\Div u,|\partial_y^{2-j}\partial_3^{j+1}q|^2)\\
&\qquad+
([\partial_y^{2-j}\partial_3^{j+1},u\cdot\nabla]q,
\partial_y^{2-j}\partial_3^{j+1}q),
\end{align*}
where the last two terms have absolute value at most $CN(t)D(t)$.
The boundary flux vanishes because $u\cdot n=0$.
The electric and time derivative terms are estimated by
\begin{align*}
&\left|\left(
\partial_y^{2-j}\partial_3^j
(\partial_3\phi-\partial_tu_3),
\partial_y^{2-j}\partial_3^{j+1}(q+w)\right)\right|\\
&\qquad\leq\epsilon
\|\partial_y^{2-j}\partial_3^{j+1}(q,w)\|^2
+C_\epsilon\bigl(\|\nabla q\|^2+\|\nabla^2u_t\|^2\bigr),
\end{align*}
using $\|\nabla^3\phi\|\leq C\|\nabla q\|$.
Also, \eqref{eq:w} and \eqref{eq:products} give
\[
\|w\|_{H^3}\leq C\|\nabla u\|_{H^3}
+\|q\Div u\|_{H^3}\leq CD(t)^{1/2},
\]
so the nonlinear terms satisfy
\begin{align*}
&\left|\left(
\partial_y^{2-j}\partial_3^j
\bigl[F_3-(2\mu+\lambda)\partial_3(q\Div u)\bigr],
\partial_y^{2-j}\partial_3^{j+1}(q+w)\right)\right|\\
&\qquad\leq C\bigl(\|F\|_{H^2}
+\|q\Div u\|_{H^3}\bigr)
\bigl(\|q\|_{H^3}+\|w\|_{H^3}\bigr)\leq CN(t)D(t).
\end{align*}
The remaining terms contain fourth
derivatives of $u$ with at least one tangential derivative.
To estimate them, write the momentum and continuity equations as
\begin{equation}\label{eq:stokes-u}
\left\{
\begin{aligned}
&-\mu\Delta u+\nabla(\gamma q-\phi)
=-u_t+(\mu+\lambda)\nabla\Div u+F,\\
&\Div u=-w-q\Div u.
\end{aligned}
\right.
\end{equation}
Apply \eqref{eq:stokes-estimate} with $m=0$ to two tangential
derivatives of \eqref{eq:stokes-u}. Substitution of the divergence
equation on the right-hand side yields
\[
\|\nabla^2\partial_y^2u\|
\leq C\bigl(\|\partial_y^2u_t\|
+\|\nabla\partial_y^2w\|
+\|F\|_{H^2}+\|q\Div u\|_{H^3}\bigr).
\]
The purely tangential derivative of $w$ satisfies
\[
\|\partial_y^3w\|
\leq C\|\nabla\partial_y^3u\|
+\|q\Div u\|_{H^3}
\leq C\|\nabla\partial_y^3u\|+CN(t)D(t)^{1/2}.
\]
Separating the tangential and normal derivatives in
$\nabla\partial_y^2w$ therefore gives
\begin{equation}\label{eq:mixed-u-two}
\begin{aligned}
\|\nabla^2\partial_y^2u\|^2
\leq C\bigl(
\|\nabla\partial_y^3u\|^2
+\|\partial_y^2\partial_3w\|^2
+\|\nabla^2u_t\|^2\bigr)+CN(t)^2D(t).
\end{aligned}
\end{equation}
Similarly, applying \eqref{eq:stokes-estimate} with $m=1$ to
one tangential derivative gives
\[
\|\nabla^3\partial_yu\|
\leq C\bigl(\|\nabla\partial_yu_t\|
+\|\nabla^2\partial_yw\|
+\|F\|_{H^2}+\|q\Div u\|_{H^3}\bigr).
\]
The derivatives of $w$ on the right have zero, one, or two
normal derivatives. Hence
\begin{equation}\label{eq:mixed-u-one}
\begin{aligned}
\|\nabla^3\partial_yu\|^2
\leq C\bigl(
\|\nabla\partial_y^3u\|^2
+\|\partial_y^2\partial_3w\|^2
+\|\partial_y\partial_3^2w\|^2
+\|\nabla^2u_t\|^2\bigr)+CN(t)^2D(t).
\end{aligned}
\end{equation}
In the normal momentum equation, the differentiated viscous
terms are bounded as follows:
\begin{align*}
\|\partial_y^2(\Delta_yu_3-\partial_3\Div_yu_y)\|
&\leq C\|\nabla\partial_y^3u\|,\\
\|\partial_y\partial_3(\Delta_yu_3-\partial_3\Div_yu_y)\|
&\leq C\|\nabla^2\partial_y^2u\|,\\
\|\partial_3^2(\Delta_yu_3-\partial_3\Div_yu_y)\|
&\leq C\|\nabla^3\partial_yu\|.
\end{align*}
Pairing these terms with
$\partial_y^{2-j}\partial_3^{j+1}(q+w)$, and choosing
$\epsilon$ small relative to $2\mu+\lambda$ and $\gamma$,
we absorb all terms multiplied by $\epsilon$ into the
positive density and material-derivative terms on the left-hand side. By \eqref{eq:mixed-u-two} and
\eqref{eq:mixed-u-one}, this proves, for $j=0,1,2$,
\begin{equation}\label{eq:normal-triangular}
\begin{aligned}
&\frac{2\mu+\lambda+\gamma}{2}\frac{\dd}{\dd t}
\|\partial_y^{2-j}\partial_3^{j+1}q\|^2
+c\|\partial_y^{2-j}\partial_3^{j+1}(q,w)\|^2\\
&\quad\leq C\|\nabla\partial_y^3u\|^2
+C\sum_{k=0}^{j-1}
\|\partial_y^{2-k}\partial_3^{k+1}(q,w)\|^2
+C\bigl(\|\nabla q\|^2+\|\nabla^2u_t\|^2\bigr)+CN(t)D(t),
\end{aligned}
\end{equation}
where the sum is zero when $j=0$. The third-order derivatives
in the sum have strictly fewer normal derivatives than those on
the left.

Multiply \eqref{eq:normal-triangular} by $\eta^j$ and sum over
$j=0,1,2$. For each $k<2$,
$\sum_{j=k+1}^2\eta^j\leq\eta^{k+1}/(1-\eta)$.
Choosing $\eta>0$ so that $C\eta/(1-\eta)\leq c/2$
therefore absorbs the terms with $k<j$ into the weighted
normal dissipation. We then add $M_4$ times
\eqref{eq:tangential-energy}, with $M_4$ sufficiently large, to obtain
\begin{equation}\label{eq:H-preliminary}
\begin{aligned}
&\frac{\dd}{\dd t}\mathcal H
+c\sum_{1\leq|\beta|\leq3}\|\nabla\partial_y^\beta u\|^2
+c\sum_{j=0}^2\|\partial_y^{2-j}\partial_3^{j+1}(q,w)\|^2\\
&\qquad\leq C\bigl(\|\nabla q\|^2+\|\nabla^2u_t\|^2\bigr)
+CN(t)D(t).
\end{aligned}
\end{equation}
The constants $\eta$ and $M_4$ are independent of $T$.

It remains to recover the full spatial norms. The preceding bound
for $\partial_y^3w$ and the normal terms give
\begin{equation}\label{eq:w-third}
\|\nabla^3w\|^2
\leq C\left(
\|\nabla\partial_y^3u\|^2
+\sum_{j=0}^2\|\partial_y^{2-j}\partial_3^{j+1}w\|^2
\right)+CN(t)^2D(t).
\end{equation}
Apply \eqref{eq:stokes-estimate} with $m=2$ to
\eqref{eq:stokes-u}. Before using \eqref{eq:w-third}, it gives
\begin{align*}
\|\nabla^4u\|+\|\nabla^3(\gamma q-\phi)\|
&\leq C\bigl(\|\nabla^2u_t\|+\|\nabla^3w\|\\
&\qquad+\|F\|_{H^2}+\|q\Div u\|_{H^3}\bigr).
\end{align*}
Thus \eqref{eq:products} and \eqref{eq:w-third} imply
\begin{equation}\label{eq:u-fourth}
\begin{aligned}
\|\nabla^4u\|^2+\|\nabla^3(\gamma q-\phi)\|^2
\leq C\left(
\|\nabla\partial_y^3u\|^2
+\sum_{j=0}^2\|\partial_y^{2-j}\partial_3^{j+1}w\|^2
+\|\nabla^2u_t\|^2\right)+CN(t)^2D(t).
\end{aligned}
\end{equation}
Since $\|\nabla^3\phi\|\leq C\|\nabla q\|$, the interpolation
inequalities
\begin{equation}\label{eq:interpolation}
\begin{aligned}
&\|\nabla q\|^2\leq\epsilon\|\nabla^3q\|^2
+C_\epsilon\|q\|^2,\\
&\|\nabla^2u\|^2+\|\nabla^3u\|^2
\leq\epsilon\|\nabla^4u\|^2+C_\epsilon\|\nabla u\|^2
\end{aligned}
\end{equation}
imply
\begin{equation}\label{eq:full-spatial-recovery}
\begin{aligned}
\|q\|_{H^3}^2+\|\nabla u\|_{H^3}^2
\leq C\biggl(&
\sum_{1\leq|\beta|\leq3}\|\nabla\partial_y^\beta u\|^2
+\sum_{j=0}^2\|\partial_y^{2-j}\partial_3^{j+1}(q,w)\|^2\\
&+\|\nabla^2u_t\|^2+\|q\|^2+\|\nabla u\|^2\biggr)+CN(t)^2D(t).
\end{aligned}
\end{equation}
The part of $CN(t)^2D(t)$ containing spatial norms is absorbed for
small $\delta$. Its remaining part is bounded by $CN(t)D(t)$.
Use \eqref{eq:full-spatial-recovery} to replace a fixed part of
the dissipation in \eqref{eq:H-preliminary} by
$\|q\|_{H^3}^2+\|\nabla u\|_{H^3}^2$. This introduces only
the lower-order terms already displayed in
\eqref{eq:full-spatial-recovery}. The first inequality in
\eqref{eq:interpolation}, with $\epsilon$ sufficiently small,
then absorbs $\|\nabla q\|^2$ into the density dissipation.
This proves \eqref{eq:spatial-energy}.
\end{proof}

We are now in a position to combine the preceding estimates and
complete the a priori bound.

\begin{proposition}\label{prop:apriori}
There exist $\delta>0$ and $C>0$, independent of $T$, such that
every solution satisfying \eqref{eq:apriori-assumption} obeys
\begin{equation}\label{eq:apriori-conclusion}
N(t)^2+\int_0^tD(s)\dd s\leq CN(0)^2,
\qquad 0\leq t\leq T,
\end{equation}
provided $N(0)$ is sufficiently small.
\end{proposition}

\begin{proof}
We first estimate the terms of indefinite sign in
$\mathcal E_2$ and $(u,\nabla q)$.
Taking the inner product of the momentum equation with $u$ yields
\[
a(u,u)=-(u_t,u)+\gamma(q,\Div u)+(\nabla\phi,u)+(F,u).
\]
The nonlinear terms satisfy
\begin{align*}
\left|\left(-u\cdot\nabla u+\frac{q}{1+q}Au,u\right)\right|
&\leq CN(t)\|\nabla u\|^2,\\
\left|\left(((1+q)^{\gamma-2}-1)\nabla q,u\right)\right|
&\leq CN(t)\|q\|\|\nabla u\|.
\end{align*}
For the viscous term in the first line, one integrates once and
uses $\|\nabla q\|_{L^3}\|u\|_{L^6}\leq CN(t)\|\nabla u\|$.
For the second line, one uses the pressure primitive in
\eqref{eq:pressure-primitive}. Young's inequality and the
smallness of $N(t)$ therefore give
\begin{equation}\label{eq:first-velocity-control}
a(u,u)\leq C\bigl(
\|q\|^2+\|u\|^2+\|\nabla\phi\|^2+\|u_t\|^2\bigr).
\end{equation}
At each fixed time, we have
\[
\|F\|\leq CN(t)\bigl(\|\nabla u\|+\|Au\|+\|\nabla q\|\bigr),
\]
so the momentum equation, after absorbing $CN(t)\|Au\|$, yields
\begin{equation}\label{eq:Au-control}
\|Au\|\leq C\bigl(
\|u_t\|+\|\nabla q\|+\|\nabla\phi\|
+N(t)\|\nabla u\|\bigr).
\end{equation}
Consequently, for every $\epsilon>0$,
\begin{equation}\label{eq:cross-control}
\begin{aligned}
|(\QQ u,Au)|+a(u,u)+|(u,\nabla q)|
\leq{}&\epsilon\|\nabla^3q\|^2\\
&+C_\epsilon\bigl(
\|q\|^2+\|u\|^2+\|\nabla\phi\|^2+\|u_t\|^2\bigr).
\end{aligned}
\end{equation}
Indeed, $\|\QQ u\|\leq\|u\|$ and
\eqref{eq:Au-control} first bound the two mixed products by
$C(\|u\|^2+\|u_t\|^2+\|\nabla\phi\|^2+\|\nabla q\|^2)
+CN(t)^2\|\nabla u\|^2$.
The last term is controlled by \eqref{eq:first-velocity-control},
and interpolation gives
\[
\|\nabla q\|^2\leq\epsilon\|\nabla^3q\|^2+C_\epsilon\|q\|^2.
\]

We now combine the preceding estimates. Let
\begin{equation}\label{eq:combined-energy}
\mathcal E
=\mathcal H+M_2\mathcal E_2+M_3(u,\nabla q)
+M_1\mathcal E_1+M_0\mathcal E_0.
\end{equation}
Choose $M_2$ and $M_3$ sufficiently large so that
\eqref{eq:gradient-time-energy} and \eqref{eq:density-energy}
absorb the terms $C\|\nabla^2u_t\|^2$ and $C\|q\|^2$ in
\eqref{eq:spatial-energy}. After these choices, the remaining
linear error is bounded by a fixed multiple of
$\|\nabla u\|_{H^1}^2$. Using
\[
\|\nabla u\|_{H^1}^2
\leq\epsilon\|\nabla u\|_{H^3}^2
+C_\epsilon\|\nabla u\|^2,
\]
we absorb its higher derivatives into the left-hand side of
\eqref{eq:spatial-energy}. Choose $M_1$ large enough to control
the $\|u_t\|^2$ terms in \eqref{eq:cross-control}, and then
choose $M_0$ large enough to control the remaining lower-order
terms and absorb the error $C\|\nabla u\|^2$.

The energy $\mathcal H$ controls every third derivative of $q$:
the pure tangential derivatives occur in the second line of
\eqref{eq:H}, and those with a normal derivative occur in the
first line. Thus \eqref{eq:cross-control}, with $\epsilon$
chosen after $M_2$ and $M_3$, proves positivity of
\eqref{eq:combined-energy}. To recover the full $H^3$ norm of
$u$, we first estimate one derivative of $F$ directly from
\eqref{eq:F}:
\begin{equation}\label{eq:F-gradient-pointwise}
\|\nabla F\|\leq CN(t)\bigl(\|u\|_{H^3}+\|q\|_{H^2}\bigr).
\end{equation}
Here the differentiated viscous coefficient is controlled in
$L^\infty$, while $\nabla Au$ is controlled by $\nabla^3u$;
the convective and pressure terms follow from the product estimates.
Applying \eqref{eq:lame-estimate} to the momentum equation and
using $\|\nabla^2\phi\|\leq C\|q\|$, we obtain
\begin{equation}\label{eq:velocity-H3}
\|\nabla^3u\|
\leq C\bigl(\|\nabla u_t\|+\|q\|_{H^2}\bigr)
+CN(t)\bigl(\|u\|_{H^3}+\|q\|_{H^2}\bigr).
\end{equation}
Interpolation and the smallness of $N(t)$ imply
\[
\|u\|_{H^3}
\leq C\bigl(\|u\|+\|\nabla u_t\|+\|q\|_{H^2}\bigr).
\]
Together with Lemma~\ref{lem:poisson},
\eqref{eq:field-time-bound} and \eqref{eq:qt-estimates},
this proves
\begin{equation}\label{eq:energy-equivalence}
cN(t)^2\leq\mathcal E(t)\leq CN(t)^2.
\end{equation}

Adding \eqref{eq:basic-energy}, \eqref{eq:time-energy},
\eqref{eq:density-energy}, \eqref{eq:gradient-time-energy} and
\eqref{eq:spatial-energy} with the chosen weights gives
\begin{equation}\label{eq:closed-energy}
\frac{\dd}{\dd t}\mathcal E+cD(t)\leq CN(t)D(t).
\end{equation}
Choose $\delta$ in \eqref{eq:apriori-assumption} small enough
that $C\delta\leq c/2$. Integrating \eqref{eq:closed-energy},
we obtain
\begin{equation}\label{eq:integrated-energy}
\mathcal E(t)+\frac c2\int_0^tD(s)\dd s\leq\mathcal E(0).
\end{equation}
The equivalence \eqref{eq:energy-equivalence} now yields
\eqref{eq:apriori-conclusion}.
\end{proof}

\section{Uniform estimates in exterior domains}\label{sec:exterior}

We next consider \eqref{eq:NSP} in a fixed smooth connected exterior
domain $\Omega=\R^3\setminus\overline{\mathcal O}$, where $\mathcal O$
is bounded. The boundary conditions are
\begin{equation}\label{eq:ext-boundary}
u=0,\qquad \partial_n\phi=0\quad\text{on }\partial\Omega,
\qquad \phi(t)\in\dot H^1(\Omega).
\end{equation}
We use the same notation as in Sections~\ref{sec:preliminaries}--\ref{sec:estimates}.

\begin{lemma}\label{lem:ext-low}
Assume \eqref{eq:apriori-assumption}. The definitions
\eqref{eq:E0}, \eqref{eq:E1} and \eqref{eq:E2} of
$\mathcal E_0,\mathcal E_1,\mathcal E_2$ are retained.
The estimates \eqref{eq:qt-estimates}, \eqref{eq:products} and
\eqref{eq:Ft-bound} hold in the exterior domain, and
\begin{align}
&\frac{\dd}{\dd t}\mathcal E_0+a(u,u)=0,
\label{eq:ext-basic}
\\
&\frac{\dd}{\dd t}\mathcal E_1+c\|\nabla u_t\|^2
\leq CN(t)D(t),
\label{eq:ext-time}
\\
&\frac{\dd}{\dd t}(u,\nabla q)+c\|q\|_{H^1}^2
\leq C\|\nabla u\|_{H^1}^2+CN(t)D(t),
\label{eq:ext-density}
\\
&\frac{\dd}{\dd t}\mathcal E_2+c\|\nabla^2u_t\|^2
\leq C\bigl(\|\nabla u\|_{H^1}^2+\|\nabla u_t\|^2\bigr)
+CN(t)D(t).
\label{eq:ext-gradient-time}
\end{align}
\end{lemma}

\begin{proof}
The product estimates preceding Lemma~\ref{lem:basic} use only the
Sobolev inequalities and remain valid. The proofs of
Lemmas~\ref{lem:basic}, \ref{lem:time} and \ref{lem:density}
use $u=0$, $\partial_n\phi=0$, and the electric estimate
$\|\nabla\phi_t\|_{L^6}\leq C\|q_t\|$.
Lemma~\ref{lem:ext-poisson} therefore gives exactly the same
identities and estimates.

The projection identity \eqref{eq:projection-energy-identity}
also remains valid. Its derivation uses the $L^2$ orthogonality
of $\PP,\QQ$, the identity \eqref{eq:field-time}, and
$(u_t,Au)=a(u,u_t)$. It does not require commutation of a projection
with $A$ or with local tangential derivatives.
Thus \eqref{eq:gradient-time-identity} gives
\[
\frac{\dd}{\dd t}\mathcal E_2+c\|Au_t\|^2
\leq C\|\nabla u\|_{H^1}^2+CN(t)D(t).
\]
Using \eqref{eq:ext-lame} introduces the additional term
$C\|\nabla u_t\|^2$ in \eqref{eq:ext-gradient-time}.
This term is controlled by \eqref{eq:ext-time}.
\end{proof}

Use the coordinates and operators in Lemma~\ref{lem:normal-coordinates}.

\begin{lemma}\label{lem:ext-tangential}
Under \eqref{eq:apriori-assumption}, for every $\epsilon>0$,
\begin{equation}\label{eq:ext-tangential-energy}
\begin{aligned}
&\frac12\frac{\dd}{\dd t}\sum_L
\bigl(\gamma\|Lq\|^2+\|Lu\|^2+\|\nabla L\phi\|^2\bigr)
+c\sum_L\|\nabla Lu\|^2
\\
&\quad\leq
\epsilon\bigl(\|q\|_{H^3}^2+\|\nabla u\|_{H^3}^2\bigr)
+C_\epsilon\bigl(\|q\|^2+\|\nabla u\|^2\bigr)
+CN(t)D(t).
\end{aligned}
\end{equation}
Here $\sum_L$ denotes the finite sum over
\eqref{eq:localized-operators}.
\end{lemma}

\begin{proof}
Applying $L$ to the fluid equations and testing by
$\gamma Lq$ and $Lu$, respectively, gives
\begin{align}
&\frac12\frac{\dd}{\dd t}
\bigl(\gamma\|Lq\|^2+\|Lu\|^2\bigr)+a(Lu,Lu)
\nonumber
\\
&\quad=(L\nabla\phi,Lu)+(LF,Lu)
-\gamma(L\Div(qu),Lq)+([A,L]u,Lu)
\nonumber
\\
&\qquad\quad-\gamma([L,\nabla]q,Lu)
-\gamma([L,\Div]u,Lq).
\label{eq:ext-tangential-identity}
\end{align}
The transport decomposition, pressure primitive and viscous expansion
in Lemma~\ref{lem:tangential} apply to the nonlinear terms.
Differentiating a chart coefficient or a cutoff produces a bounded
coefficient and a lower-order product. The sum of these nonlinear
terms is bounded by $CN(t)D(t)$.

For the linear geometric terms, the commutator orders stated after
Lemma~\ref{lem:normal-coordinates} give
\begin{align*}
&\sum_L|([A,L]u,Lu)|
\leq C\|\nabla u\|_{H^3}\|\nabla u\|_{H^2},
\\
&\sum_L\bigl(|([L,\nabla]q,Lu)|+|([L,\Div]u,Lq)|\bigr)
\leq C\|q\|_{H^3}\|\nabla u\|_{H^2}.
\end{align*}
Terms involving $u$ without a derivative are supported in a fixed
bounded set and satisfy \eqref{eq:ext-local-sobolev}. Young's
inequality and Sobolev interpolation therefore bound the preceding
sum by
\begin{equation}\label{eq:ext-geometric-error}
\epsilon\bigl(\|q\|_{H^3}^2+\|\nabla u\|_{H^3}^2\bigr)
+C_\epsilon\bigl(\|q\|^2+\|\nabla u\|^2\bigr).
\end{equation}

We transform the electric force by the continuity and Poisson
equations. Since $Lu=0$ and $\partial_nL\phi_t=0$ on the boundary,
\begin{equation}\label{eq:ext-electric-transformation}
\begin{aligned}
&(L\nabla\phi,Lu)
={}-(L\phi,\Div Lu)+([L,\nabla]\phi,Lu)
\\
&\quad ={}(L\phi,Lq_t)+(L\phi,L\Div(qu))
+(L\phi,[L,\Div]u)
\\
&\quad +([L,\nabla]\phi,Lu)
\\
&\quad ={}(L\phi,\Delta L\phi_t)+(L\phi,[L,\Delta]\phi_t)
-(\nabla L\phi,L(qu))
\\
&\quad +(L\phi,[L,\Div](qu))+(L\phi,[L,\Div]u)
+([L,\nabla]\phi,Lu)
\\
&\quad ={}-\frac12\frac{\dd}{\dd t}\|\nabla L\phi\|^2
-(\nabla L\phi,L(qu))+\mathcal R_L,
\end{aligned}
\end{equation}
where
\begin{equation}\label{eq:ext-electric-remainder}
\begin{aligned}
&\mathcal R_L={}(L\phi,[L,\Delta]\phi_t)
+(L\phi,[L,\Div](qu))
\\
&\quad +(L\phi,[L,\Div]u)+([L,\nabla]\phi,Lu).
\end{aligned}
\end{equation}
In the third line of \eqref{eq:ext-electric-transformation} we used
$L\Delta\phi_t=\Delta L\phi_t+[L,\Delta]\phi_t$ and
$L\Div(qu)=\Div L(qu)+[L,\Div](qu)$.
The integration involving $L(qu)$ has zero boundary term because
$L(qu)=0$ on $\partial\Omega$.

Let $K$ contain all commutator supports and let $L$ have order $k$,
$1\leq k\leq3$. Lemma~\ref{lem:ext-poisson} and
\eqref{eq:ext-local-sobolev} imply
\begin{equation}\label{eq:ext-electric-commutators}
\begin{aligned}
&\|\nabla L\phi\|\leq C\|q\|_{H^{k-1}},\qquad
\|L\phi\|_{L^2(K)}+\|[L,\nabla]\phi\|
\leq C\|q\|_{H^{\max\{k-2,0\}}},
\\
&\|[L,\Delta]\phi_t\|\leq C\|q_t\|_{H^{k-1}},
\qquad \|[L,\Div]u\|\leq C\|\nabla u\|_{H^{k-1}},
\\
&\|[L,\Div](qu)\|+\|L(qu)\|
\leq CN(t)D(t)^{1/2}.
\end{aligned}
\end{equation}
For example, $[L,\Delta]\phi_t$ contains derivatives of $\phi_t$
of orders $1$ through $k+1$. Its first derivatives on $K$ are
bounded by $C\|q_t\|$, and its higher derivatives are bounded by
$C\|q_t\|_{H^{k-1}}$. Thus all its terms have the stated bound.
The product estimate in the last line follows by placing $u$ or
$q$ in $L^\infty$, and using \eqref{eq:ext-local-sobolev} for
local zero-order velocity terms.

Consequently,
\begin{align*}
&\sum_L\bigl(|\mathcal R_L|
+|(\nabla L\phi,L(qu))|\bigr)
\leq C\|q\|_{H^2}
\bigl(\|q_t\|_{H^2}+\|\nabla u\|_{H^2}\bigr)
+CN(t)D(t)
\\
&\quad \leq C\|q\|_{H^2}\|\nabla u\|_{H^2}+CN(t)D(t)
\\
&\quad \leq\epsilon\bigl(\|q\|_{H^3}^2+\|\nabla u\|_{H^3}^2\bigr)
+C_\epsilon\bigl(\|q\|^2+\|\nabla u\|^2\bigr)
+CN(t)D(t).
\end{align*}
Here \eqref{eq:qt-estimates} bounds
$\|q_t\|_{H^2}$ by
$C\|\nabla u\|_{H^2}+CN(t)\|q\|_{H^3}$.
Substitution of \eqref{eq:ext-electric-transformation} in
\eqref{eq:ext-tangential-identity}, followed by summation and
\eqref{eq:ext-geometric-error}, proves
\eqref{eq:ext-tangential-energy}.
\end{proof}

Applying Lemma~\ref{lem:ext-stokes} to \eqref{eq:stokes-u} gives
\begin{equation}\label{eq:ext-stokes-u}
\begin{aligned}
&\|\nabla^4u\|+\|\nabla^3(\gamma q-\phi)\|
\leq C\bigl(\|\nabla u_t\|_{H^1}+\|\nabla^3w\|
\\
&\quad +\|\nabla u\|_{H^2}+\|q\|_{H^2}\bigr)
+CN(t)D(t)^{1/2}.
\end{aligned}
\end{equation}

For \eqref{eq:stokes-u}, use
$\Div u=-w-q\Div u$ and \eqref{eq:products}.
The localized zero-order time derivative satisfies
\[
\|u_t\|_{L^2(K)}\leq C_K\|u_t\|_{L^6}
\leq C_K\|\nabla u_t\|.
\]
Also,
\[
\|\nabla(\gamma q-\phi)\|_{H^1(K)}
\leq C\|q\|_{H^2}
\]
by Lemma~\ref{lem:ext-poisson}.
These bounds prove \eqref{eq:ext-stokes-u}.

\begin{lemma}\label{lem:ext-mixed}
Under \eqref{eq:apriori-assumption}, the mixed derivatives satisfy
\begin{equation}\label{eq:ext-mixed-two}
\begin{aligned}
&\sum_\ell\|\chi_\ell\nabla^2\partial_y^2u\|^2
\leq{}C\sum_L\|\nabla Lu\|^2
+C\sum_\ell\|\chi_\ell\partial_y^2\partial_sw\|^2
\\
&\quad +C\bigl(\|\nabla u_t\|_{H^1}^2+\|q\|_{H^2}^2
+\|\nabla u\|_{H^2}^2\bigr)+CN(t)^2D(t),
\end{aligned}
\end{equation}
and
\begin{equation}\label{eq:ext-mixed-one}
\begin{aligned}
&\sum_\ell\|\chi_\ell\nabla^3\partial_yu\|^2
\leq{}C\sum_L\|\nabla Lu\|^2
\\
&\quad +C\sum_\ell\bigl(
\|\chi_\ell\partial_y^2\partial_sw\|^2
+\|\chi_\ell\partial_y\partial_s^2w\|^2\bigr)
\\
&\quad +C\bigl(\|\nabla u_t\|_{H^1}^2+\|q\|_{H^2}^2
+\|\nabla u\|_{H^2}^2\bigr)+CN(t)^2D(t).
\end{aligned}
\end{equation}
\end{lemma}

\begin{proof}
Apply the local $H^2$ Stokes estimate
to $\partial_y^2u$ and the local $H^3$ estimate to $\partial_yu$,
using a cutoff equal to one on the support of the smaller cutoff.
The commutators $[\Delta,\partial_y^2]u$ and
$[\partial_y^2,\nabla](\gamma q-\phi)$ contain at most three
derivatives of $u$ and two derivatives of $\gamma q-\phi$.
The commutators for $\partial_yu$, differentiated once, have the
same bounds. All these terms are included in the lower norms
displayed above.

The highest divergence data are
$\nabla\partial_y^2w$ and $\nabla^2\partial_yw$.
Their components are, respectively,
\[
(\partial_y^3w,\partial_y^2\partial_sw),\qquad
(\partial_y^3w,\partial_y^2\partial_sw,
\partial_y\partial_s^2w),
\]
up to lower-order terms. The pure tangential derivative
$\partial_y^3w$ is controlled by the corresponding fourth
derivatives of $u$ and \eqref{eq:products}, as in the proof of
\eqref{eq:mixed-u-two}. Finally, sum over a fixed finite covering.
On overlapping normal charts the distance coordinate is the same,
and a change of tangential coordinates preserves the number of
normal derivatives. A partition of unity therefore controls the
larger-cutoff terms by the sums displayed above. Derivatives of
the transition functions or cutoffs contribute only lower-order
terms.

\end{proof}

Let $\nu$ be the unit normal pointing into the fluid, extended
along normal coordinate lines. Thus
$\partial_s=\nu\cdot\nabla$ and $\partial_s\nu=0$ in the collar.
Taking the normal component of the momentum equation gives
\begin{equation}\label{eq:ext-normal-momentum}
\begin{aligned}
&(2\mu+\lambda)\partial_sw+\gamma\partial_sq
={}\partial_s\phi-u_t\cdot\nu
+\mu\nu\cdot(\Delta u-\nabla\Div u)
\\
&\quad -(2\mu+\lambda)\partial_s(q\Div u)+F\cdot\nu.
\end{aligned}
\end{equation}

By \eqref{eq:coordinate-differential},
\[
\nu\cdot\nabla\Div u=\partial_s\Div u
=\nu\cdot\partial_s^2u+t^\alpha\cdot\partial_{y_\alpha}\partial_su
+(\partial_st^\alpha)\cdot\partial_{y_\alpha}u.
\]
Subtracting this identity from the normal component of the Laplacian
gives
\begin{equation}\label{eq:ext-normal-cancellation}
\begin{aligned}
&\nu\cdot(\Delta u-\nabla\Div u)
={}g^{\alpha\beta}\nu\cdot\partial_{y_\alpha y_\beta}u
-t^\alpha\cdot\partial_{y_\alpha}\partial_su
\\
&\quad +(\Div\nu)\nu\cdot\partial_su
+\bigl((\Div t^\alpha)\nu-\partial_st^\alpha\bigr)
\cdot\partial_{y_\alpha}u.
\end{aligned}
\end{equation}
The pure second normal derivative of $u$ cancels, and the last line
contains only first derivatives of $u$.
After applying $\partial_y^2$, $\partial_y\partial_s$ and
$\partial_s^2$, the highest velocity derivatives have the types
\[
(\partial_y^4u,\partial_y^3\partial_su),\qquad
(\partial_y^3\partial_su,\partial_y^2\partial_s^2u),\qquad
(\partial_y^2\partial_s^2u,\partial_y\partial_s^3u),
\]
respectively. They are controlled by the tangential dissipation,
\eqref{eq:ext-mixed-two}, and \eqref{eq:ext-mixed-one}, in that
order.

\begin{lemma}\label{lem:ext-normal}
There are constants $0<\eta<1$ and $M_4>0$ such that the energy
\begin{equation}\label{eq:ext-H}
\begin{aligned}
&\mathcal H(t)={}\frac{2\mu+\lambda+\gamma}{2}
\sum_{j=0}^2\eta^j\sum_\ell
\|\chi_\ell\partial_y^{2-j}\partial_s^{j+1}q\|^2
\\
&\quad +\frac{M_4}{2}\sum_L
\bigl(\gamma\|Lq\|^2+\|Lu\|^2+\|\nabla L\phi\|^2\bigr)
\end{aligned}
\end{equation}
satisfies
\begin{equation}\label{eq:ext-spatial-energy}
\begin{aligned}
&\frac{\dd}{\dd t}\mathcal H
+c\bigl(\|q\|_{H^3}^2+\|\nabla u\|_{H^3}^2\bigr)
\leq C\bigl(
\|\nabla u_t\|_{H^1}^2+\|q\|^2+\|\nabla u\|^2\bigr)
+CN(t)D(t).
\end{aligned}
\end{equation}
\end{lemma}

\begin{proof}
For $j=0,1,2$, apply $\partial_y^{2-j}\partial_s^j$ to
\eqref{eq:ext-normal-momentum}, and test with
$\chi_\ell^2\partial_y^{2-j}\partial_s^{j+1}(q+w)$.
The identity $w=q_t+u\cdot\nabla q$ gives
\[
\begin{aligned}
&\partial_y^{2-j}\partial_s^{j+1}w
=\partial_t\partial_y^{2-j}\partial_s^{j+1}q
+u\cdot\nabla\partial_y^{2-j}\partial_s^{j+1}q
\\
&\qquad+[\partial_y^{2-j}\partial_s^{j+1},u\cdot\nabla]q.
\end{aligned}
\]
In local coordinates the transport coefficients are $u\cdot\nu$
and $u\cdot t^\alpha$. The product rule, the bounded chart
coefficients and \eqref{eq:sobolev-commutator} give
\[
\|\chi_\ell[\partial_y^{2-j}\partial_s^{j+1},u\cdot\nabla]q\|
\leq CN(t)D(t)^{1/2}.
\]
Integration with respect to the physical volume element yields
\begin{align*}
&(\chi_\ell\partial_y^{2-j}\partial_s^{j+1}w,
\chi_\ell\partial_y^{2-j}\partial_s^{j+1}q)
\\
&\quad=\frac12\frac{\dd}{\dd t}
\|\chi_\ell\partial_y^{2-j}\partial_s^{j+1}q\|^2
-\frac12\int_\Omega\Div(\chi_\ell^2u)
|\partial_y^{2-j}\partial_s^{j+1}q|^2\dd x
\\
&\qquad+(\chi_\ell[\partial_y^{2-j}\partial_s^{j+1},u\cdot\nabla]q,
\chi_\ell\partial_y^{2-j}\partial_s^{j+1}q).
\end{align*}
The last two terms are bounded by $CN(t)D(t)$.
The quadratic terms in the tested normal momentum equation are
\[
\begin{aligned}
&(2\mu+\lambda)\|\chi_\ell\partial_y^{2-j}\partial_s^{j+1}w\|^2
+\gamma\|\chi_\ell\partial_y^{2-j}\partial_s^{j+1}q\|^2
\\
&\qquad+(2\mu+\lambda+\gamma)
(\chi_\ell\partial_y^{2-j}\partial_s^{j+1}w,
\chi_\ell\partial_y^{2-j}\partial_s^{j+1}q).
\end{aligned}
\]

For the electric term, \eqref{eq:coordinate-differential} expresses
$\partial_y^{2-j}\partial_s^{j+1}\phi$ as a sum of physical
derivatives of $\phi$ of orders one through three with fixed smooth
coefficients. Lemma~\ref{lem:ext-poisson} gives
\begin{equation}\label{eq:ext-normal-electric}
\begin{aligned}
&\|\chi_\ell\partial_y^{2-j}\partial_s^{j+1}\phi\|
\leq C\bigl(\|\nabla\phi\|_{L^2(K)}
+\|\nabla^2\phi\|+\|\nabla^3\phi\|\bigr)
\leq C\|q\|_{H^1},
\\
&|(\chi_\ell\partial_y^{2-j}\partial_s^{j+1}\phi,
\chi_\ell\partial_y^{2-j}\partial_s^{j+1}(q+w))|
\\
&\qquad\leq\epsilon
\|\chi_\ell\partial_y^{2-j}\partial_s^{j+1}(q,w)\|^2
+C_\epsilon\|q\|_{H^1}^2.
\end{aligned}
\end{equation}
The time derivative term obeys
\[
\begin{aligned}
&\|\chi_\ell\partial_y^{2-j}\partial_s^j(u_t\cdot\nu)\|
\leq C\bigl(\|\nabla^2u_t\|+\|\nabla u_t\|
+\|u_t\|_{L^2(K)}\bigr)
\leq C\|\nabla u_t\|_{H^1}.
\end{aligned}
\]
Young's inequality therefore treats its pairing in the same way.
The terms involving $F$ and $q\Div u$ are bounded by $CN(t)D(t)$
using \eqref{eq:products} and $\|w\|_{H^3}\leq CD(t)^{1/2}$.
In \eqref{eq:ext-normal-cancellation}, derivatives of coefficients
leave at most three velocity derivatives. The fourth derivatives
are estimated by Lemma~\ref{lem:ext-mixed}, in the order
$j=0,1,2$ described above. Choosing the Young parameters sufficiently
small gives
\begin{equation}\label{eq:ext-normal-triangular}
\begin{aligned}
&\frac{2\mu+\lambda+\gamma}{2}\frac{\dd}{\dd t}
\sum_\ell\|\chi_\ell\partial_y^{2-j}\partial_s^{j+1}q\|^2
+c\sum_\ell
\|\chi_\ell\partial_y^{2-j}\partial_s^{j+1}(q,w)\|^2
\\
&\quad\leq C\sum_L\|\nabla Lu\|^2
+C\sum_{k<j}\sum_\ell
\|\chi_\ell\partial_y^{2-k}\partial_s^{k+1}(q,w)\|^2
\\
&\qquad\quad
+C\bigl(\|\nabla u_t\|_{H^1}^2+\|q\|_{H^2}^2
+\|\nabla u\|_{H^2}^2\bigr)+CN(t)D(t).
\end{aligned}
\end{equation}
The normal derivative count is thus exactly the one in
\eqref{eq:normal-triangular}.

Choose $\eta$ as in Lemma~\ref{lem:normal}, then choose $M_4$ large
enough to absorb the tangential terms in
\eqref{eq:ext-normal-triangular}. Only after these choices do we
choose the small parameter in \eqref{eq:ext-tangential-energy}.
Interpolation of the lower spatial norms yields
\begin{equation}\label{eq:ext-H-preliminary}
\begin{aligned}
&\frac{\dd}{\dd t}\mathcal H
+c\sum_L\|\nabla Lu\|^2
+c\sum_{j=0}^2\sum_\ell
\|\chi_\ell\partial_y^{2-j}\partial_s^{j+1}(q,w)\|^2
\\
&\quad\leq
\epsilon\bigl(\|q\|_{H^3}^2+\|\nabla u\|_{H^3}^2\bigr)
+C_\epsilon(\|q\|^2+\|\nabla u\|^2)
+C\|\nabla u_t\|_{H^1}^2+CN(t)D(t).
\end{aligned}
\end{equation}

To recover the full spatial norms, first use the covering and
$w=-\Div u-q\Div u$ to obtain
\[
\begin{aligned}
&\|\nabla^3w\|^2\leq C\biggl(
\sum_L\|\nabla Lu\|^2
+\sum_{j=0}^2\sum_\ell
\|\chi_\ell\partial_y^{2-j}\partial_s^{j+1}w\|^2
+\|\nabla u\|_{H^2}^2\biggr)+CN(t)^2D(t).
\end{aligned}
\]
Use this in \eqref{eq:ext-stokes-u}, and write
$\gamma\nabla^3q=\nabla^3(\gamma q-\phi)+\nabla^3\phi$.
Lemma~\ref{lem:ext-poisson} and interpolation then give
\begin{equation}\label{eq:ext-full-recovery}
\begin{aligned}
&\|q\|_{H^3}^2+\|\nabla u\|_{H^3}^2
\leq C\biggl(
\sum_L\|\nabla Lu\|^2
+\sum_{j=0}^2\sum_\ell
\|\chi_\ell\partial_y^{2-j}\partial_s^{j+1}(q,w)\|^2
\\
&\quad +\|\nabla u_t\|_{H^1}^2+\|q\|^2+\|\nabla u\|^2
\biggr)+CN(t)^2D(t).
\end{aligned}
\end{equation}
For small $\delta$, the spatial part of the last term is absorbed;
its remaining part is bounded by a fixed multiple of
$\|\nabla u_t\|_{H^1}^2$. Combining
\eqref{eq:ext-full-recovery} with \eqref{eq:ext-H-preliminary}
and then choosing $\epsilon$ sufficiently small proves
\eqref{eq:ext-spatial-energy}.
\end{proof}

\begin{proposition}\label{prop:ext-apriori}
There exist $\delta>0$ and $C>0$, depending on the fixed exterior
domain and the physical parameters but not on $T$, such that every
solution satisfying \eqref{eq:regularity},
\eqref{eq:ext-boundary} and \eqref{eq:apriori-assumption} obeys
\eqref{eq:apriori-conclusion}.
\end{proposition}

\begin{proof}
Use exactly the combination \eqref{eq:combined-energy}, with
$\mathcal H$ now defined by \eqref{eq:ext-H}.
Choose $M_2,M_3$ to absorb the
$\|\nabla^2u_t\|^2$ and $\|q\|^2$ terms in
\eqref{eq:ext-spatial-energy}. Interpolate the resulting
$\|\nabla u\|_{H^1}^2$ terms as in
Proposition~\ref{prop:apriori}.
In the exterior domain, $M_1$ is also chosen large enough to
absorb the $\|\nabla u_t\|^2$ terms in
\eqref{eq:ext-gradient-time} and \eqref{eq:ext-spatial-energy}.
Finally choose $M_0$ to absorb the remaining
$\|\nabla u\|^2$ terms.

The estimates \eqref{eq:first-velocity-control},
\eqref{eq:Au-control} and \eqref{eq:cross-control} are unchanged.
The covering in \eqref{eq:ext-H} gives
\[
\|q\|_{H^3}^2\leq C(\mathcal H+\|q\|^2).
\]
Thus the same order of choices proves positivity of
$\mathcal E$: after fixing $M_2,M_3$, choose the small parameter
in \eqref{eq:cross-control}, then increase $M_1,M_0$ if needed.
For the full velocity norm, replace the half-space estimate
\eqref{eq:velocity-H3} by \eqref{eq:ext-lame-energy} applied to
the momentum equation:
\[
\|u\|_{H^3}
\leq C\bigl(\|u\|+\|u_t\|_{H^1}
+\|q\|_{H^2}+\|\nabla\phi\|\bigr)
+CN(t)\|u\|_{H^3}.
\]
Absorption proves the same energy equivalence
\eqref{eq:energy-equivalence}.
Adding \eqref{eq:ext-basic}--\eqref{eq:ext-gradient-time}
and \eqref{eq:ext-spatial-energy} with the fixed weights proves
\eqref{eq:closed-energy}. Choose $\delta$ only after all these
constants have been fixed. The integration and energy equivalence
are exactly as in \eqref{eq:integrated-energy}, proving the result.
\end{proof}

\section{Proofs of the main results}\label{sec:global}

\subsection{Global existence in the two domains}

\begin{proof}[Proof of Theorems~\ref{thm:global} and~\ref{thm:ext-global}]
Lemma~\ref{lem:local} gives a strong solution on a nontrivial
time interval. At $t=0$,
the momentum equation and \eqref{eq:compatibility} give
$u_t(0)\in H^1_0$. Lemmas~\ref{lem:poisson} and~\ref{lem:ext-poisson} and the initial
regularity imply
\begin{equation}\label{eq:initial-N}
N(0)\leq C\varepsilon_0.
\end{equation}
Choose $\varepsilon_*$ in \eqref{eq:smallness} sufficiently
small. By continuity, the local solution satisfies
\eqref{eq:apriori-assumption} on an initial interval.
Proposition~\ref{prop:apriori} in the half-space and
Proposition~\ref{prop:ext-apriori} in an exterior domain give
\begin{equation}\label{eq:uniform-N}
N(t)^2+\int_0^tD(s)\dd s\leq C\varepsilon_0^2.
\end{equation}
Decreasing $\varepsilon_*$ if necessary, the right-hand side
ensures $N(t)\leq\delta/2$ and, by Sobolev embedding,
$1/2\leq1+q(t,x)\leq3/2$. This improves the a priori assumption.
The continuation criterion therefore excludes a finite maximal
existence time.

By Lemmas~\ref{lem:poisson} and~\ref{lem:ext-poisson},
\[
\|\nabla\phi(t)\|_{H^3}
\leq C\bigl(\|\nabla\phi(t)\|+\|q(t)\|_{H^2}\bigr).
\]
The continuity equation gives, by \eqref{eq:qt-estimates},
\[
\sup_{t\geq0}\|q_t(t)\|_{H^2}^2
+\int_0^\infty\|q_t(t)\|_{H^2}^2\dd t
\leq C\varepsilon_0^2.
\]
Together with \eqref{eq:uniform-N}, these estimates prove
\eqref{eq:global-estimate} and \eqref{eq:ext-global-estimate}. On each finite interval,
the time derivative of the momentum equation reads
\[
u_{tt}=-Au_t-\gamma\nabla q_t+\nabla\phi_t+F_t.
\]
The first, second and fourth terms belong to $L^2(0,T;L^2)$
by \eqref{eq:uniform-N} and \eqref{eq:Ft-bound}; the third
belongs to $L^\infty(0,T;L^2)$ by \eqref{eq:field-time-bound}.
Hence $u_{tt}\in L^2(0,T;L^2)$. The remaining continuity
properties in \eqref{eq:regularity} follow from the local theory
on every finite interval. Local uniqueness, applied successively
on these intervals, proves uniqueness of the global solution.
\end{proof}

The additional low-frequency condition allows us to use the dissipation
in Section~\ref{sec:estimates} to obtain a decay rate. We use the momentum
$(1+q)u$ in the negative-order estimate so that the continuity equation
has no nonlinear remainder.

\subsection{Decay in the half-space}

\begin{proof}[Proof of Theorem~\ref{thm:decay}]
Let $N(t)$, $D(t)$ and $\mathcal E$ be given by \eqref{eq:ND} and
\eqref{eq:combined-energy}. The preceding proof and
\eqref{eq:closed-energy} imply
\begin{equation}\label{eq:decay-energy}
cN(t)^2\leq\mathcal E\leq CN(t)^2,\qquad
\frac{\dd}{\dd t}\mathcal E+cD(t)\leq0,\qquad
\sup_{t\geq0}N(t)^2+\int_0^\infty D(t)\dd t
\leq C\varepsilon_0^2.
\end{equation}

Set $m=(1+q)u$. Since
\[
q\nabla\phi
=\Div\left(\nabla\phi\otimes\nabla\phi
-\frac12|\nabla\phi|^2I\right),
\]
the equations can be written in the conservative form
\begin{equation}\label{eq:conservative-decay}
\left\{
\begin{aligned}
&q_t+\Div m=0,
\\
&m_t+Am+\gamma\nabla q-\nabla\phi=\Div S,
\\
&\nabla\phi_t=-\QQ m,
\end{aligned}
\right.
\end{equation}
where
\begin{equation}\label{eq:decay-stress}
\begin{aligned}
&S={}-m\otimes u-\bigl(p(1+q)-p(1)-\gamma q\bigr)I
\\
&\quad +\nabla\phi\otimes\nabla\phi-\frac12|\nabla\phi|^2I
\\
&\quad -\mu\nabla(qu)-(\mu+\lambda)\Div(qu)I.
\end{aligned}
\end{equation}
Here $I$ denotes the identity matrix,
$(\nabla v)_{ij}=\partial_jv_i$, and
$(\Div S)_i=\sum_j\partial_jS_{ij}$.
The last line of \eqref{eq:decay-stress} accounts for
$Am-Au=A(qu)$. The boundary condition is $m=0$ on $\partial\Omega$.

Lemma~\ref{lem:poisson} and the smallness of $N(t)$ give
\[
\|m\|+\|u\|+\|\nabla\phi\|\leq CN(t),\qquad
\|\nabla m\|+\|\nabla u\|+\|\nabla^2\phi\|
\leq CD(t)^{1/2}.
\]
For the momentum, this follows from
$\nabla m=(1+q)\nabla u+u\otimes\nabla q$.
By Lemma~\ref{lem:tangential-negative},
\begin{equation}\label{eq:negative-stress}
\|\Lambda_y^{-s}S\|
\leq C_sN(t)^{s+\frac12}D(t)^{(\frac32-s)/2},
\qquad \frac12\leq s<1.
\end{equation}
Indeed, the convective and electric terms satisfy
\[
\begin{aligned}
&\|\Lambda_y^{-s}(m\otimes u)\|
\leq C_s\|u\|^{s-\frac12}
\|\nabla u\|^{\frac32-s}\|m\|,
\\
&\|\Lambda_y^{-s}(\nabla\phi\otimes\nabla\phi)\|
\leq C_s\|\nabla\phi\|^{s+\frac12}
\|\nabla^2\phi\|^{\frac32-s}.
\end{aligned}
\]
For the pressure remainder, use the factorization
\[
p(1+q)-p(1)-\gamma q
=q\int_0^1\bigl(p'(1+\theta q)-p'(1)\bigr)\dd\theta.
\]
The $L^2$ norm of the second factor is bounded by $C\|q\|$.
Applying \eqref{eq:negative-product} to these two factors gives the
bound $C_s\|q\|^{s+\frac12}\|\nabla q\|^{\frac32-s}$.
For the viscous corrections, write
$\nabla(qu)=q\nabla u+u\otimes\nabla q$ and use
\[
\begin{aligned}
&\|\Lambda_y^{-s}(q\nabla u)\|
\leq C_s\|q\|^{s-\frac12}
\|\nabla q\|^{\frac32-s}\|\nabla u\|,
\\
&\|\Lambda_y^{-s}(u\otimes\nabla q)\|
\leq C_s\|u\|^{s-\frac12}
\|\nabla u\|^{\frac32-s}\|\nabla q\|.
\end{aligned}
\]
In these bounds the last factor is controlled by $CN(t)$, while the
factor with exponent $\frac32-s$ is controlled by
$CD(t)^{(\frac32-s)/2}$. This proves \eqref{eq:negative-stress}.

Define the negative-order energy
\[
\mathcal L_s(t)=\frac12\left(
\gamma\|\Lambda_y^{-s}q(t)\|^2
+\|\Lambda_y^{-s}m(t)\|^2
+\|\Lambda_y^{-s}\nabla\phi(t)\|^2\right).
\]
The operators $\Lambda_y^{-s}$ and $\QQ$ commute, since the
characterization \eqref{eq:projection-definition} is invariant under
tangential translations. Apply $\Lambda_y^{-s}$ to
\eqref{eq:conservative-decay}, and take the inner products with
$\gamma\Lambda_y^{-s}q$, $\Lambda_y^{-s}m$ and
$\Lambda_y^{-s}\nabla\phi$, respectively.
The pressure terms cancel after integration by parts, and the electric
contribution to the fluid energy is
$(\Lambda_y^{-s}\nabla\phi,\Lambda_y^{-s}m)$ on the right-hand side.
On the other hand, since $\QQ$ is an orthogonal projection and
$\QQ\nabla\phi=\nabla\phi$,
\[
\begin{aligned}
&\frac12\frac{\dd}{\dd t}\|\Lambda_y^{-s}\nabla\phi\|^2
=-(\QQ\Lambda_y^{-s}m,\Lambda_y^{-s}\nabla\phi)
\\
&\quad =-(\Lambda_y^{-s}m,\Lambda_y^{-s}\nabla\phi).
\end{aligned}
\]
Adding these identities cancels the electric terms. Since
$\Lambda_y^{-s}m=0$ on $\partial\Omega$, we obtain
\[
\frac{\dd}{\dd t}\mathcal L_s
+a(\Lambda_y^{-s}m,\Lambda_y^{-s}m)
=-(\Lambda_y^{-s}S,\nabla\Lambda_y^{-s}m).
\]
Using $(\epsilon^2-\Delta_y)^{-s/2}$ in the same calculation,
we apply Young's inequality and \eqref{eq:negative-stress}, integrate
in time, and let $\epsilon\downarrow0$ by monotone convergence to obtain
\begin{equation}\label{eq:negative-energy-integral}
\begin{aligned}
&\mathcal L_s(t)
+c\int_0^t\|\nabla\Lambda_y^{-s}m(\tau)\|^2\dd\tau
\leq \mathcal L_s(0)
+C_s\int_0^tN(\tau)^{2s+1}D(\tau)^{\frac32-s}\dd\tau.
\end{aligned}
\end{equation}
The right-hand side is finite on each finite interval by
\eqref{eq:decay-energy}. Moreover,
\begin{equation}\label{eq:negative-momentum-transfer}
\|\Lambda_y^{-s}(qu)\|
\leq C_s\|q\|^{s-\frac12}
\|\nabla q\|^{\frac32-s}\|u\|
\leq C_sN(t)^2.
\end{equation}
Thus \eqref{eq:negative-data} implies
$\mathcal L_s(0)\leq C_s(M_{s,0}^2+\varepsilon_0^4)$.
Since $u=m-qu$, we also have
\[
\begin{aligned}
&\|\Lambda_y^{-s}u(t)\|
\leq \|\Lambda_y^{-s}m(t)\|
+\|\Lambda_y^{-s}(qu)(t)\|
\leq \sqrt{2\mathcal L_s(t)}+C_sN(t)^2,
\\
&\|\Lambda_y^{-s}\nabla\phi(t)\|
\leq\sqrt{2\mathcal L_s(t)}.
\end{aligned}
\]

We next relate a uniform bound for $\mathcal L_s$ to the decay of the
full energy. Suppose that $\sup_{t\geq0}\mathcal L_s(t)<\infty$.
The preceding bounds then control both negative norms uniformly.
The regularity in Theorem~\ref{thm:global} also gives
$u,\nabla\phi\in H^1$, so \eqref{eq:tangential-interpolation} applies
to each of them. Together with Lemma~\ref{lem:poisson}, it yields
\begin{equation}\label{eq:decay-low-order}
\begin{aligned}
&\|u\|^2+\|\nabla\phi\|^2
\leq C_s\bigl(\mathcal L_s(t)+\varepsilon_0^4\bigr)^{1/(1+s)}
\bigl(\|\nabla_yu\|^2
+\|\nabla_y\nabla\phi\|^2\bigr)^{s/(1+s)}
\\
&\quad \leq C_sD(t)^{s/(1+s)}.
\end{aligned}
\end{equation}
Here and in the following conditional decay estimate, the constants
may also depend on $\sup_{t\geq0}\mathcal L_s(t)$.
The $L^2$ norm of $u_t$ is estimated from the momentum equation.
In fact, directly from \eqref{eq:F},
\[
\|F\|\leq CN(t)\bigl(\|\nabla u\|_{H^1}+\|\nabla q\|\bigr),
\]
and consequently
\[
\|u_t\|^2
\leq C\bigl(\|\nabla u\|_{H^1}^2
+\|q\|_{H^1}^2+\|\nabla\phi\|^2\bigr).
\]
It follows that
\[
N(t)^2\leq C\bigl(
\|q\|_{H^3}^2+\|\nabla u\|_{H^2}^2+\|\nabla u_t\|^2
+\|u\|^2+\|\nabla\phi\|^2\bigr).
\]
The sum of the first three terms is bounded both by $D(t)$ and by
$CN(t)^2\leq C\varepsilon_0^2$. Hence it is bounded by
$C\varepsilon_0^{2/(1+s)}D(t)^{s/(1+s)}$.
Together with \eqref{eq:decay-low-order}, this proves
$\mathcal E\leq C_sD(t)^{s/(1+s)}$. We can therefore use
\eqref{eq:decay-energy} to obtain
\[
\frac{\dd}{\dd t}\mathcal E+c_s\mathcal E^{1+1/s}\leq0,
\]
and integration gives
\begin{equation}\label{eq:energy-decay-implication}
\mathcal E(t)\leq C_s(1+t)^{-s},\qquad
N(t)\leq C_s(1+t)^{-s/2},
\end{equation}
whenever $\mathcal L_s$ is uniformly bounded.

We now establish the required uniform bound.
For $\frac12<s<1$, Plancherel's theorem and H\"older's inequality give
\[
\begin{aligned}
&M_{1/2,0}
=\|\Lambda_y^{-1/2}(q_0,u_0,\nabla\phi_0)\|
\\
&\quad \leq
\|(q_0,u_0,\nabla\phi_0)\|^{1-\frac1{2s}}
M_{s,0}^{\frac1{2s}}
\leq \varepsilon_0^{1-\frac1{2s}}M_{s,0}^{\frac1{2s}}<\infty.
\end{aligned}
\]
For $s=\frac12$, this finiteness is precisely
\eqref{eq:negative-data}. In either case,
\eqref{eq:negative-momentum-transfer} at $t=0$ implies
$\mathcal L_{1/2}(0)\leq C(M_{1/2,0}^2+\varepsilon_0^4)$.
Applying \eqref{eq:negative-energy-integral} with exponent $\frac12$
and using \eqref{eq:decay-energy}, we obtain
\[
\sup_{t\geq0}\mathcal L_{1/2}(t)
\leq \mathcal L_{1/2}(0)+C\int_0^\infty N(t)^2D(t)\dd t
\leq C(M_{1/2,0}^2+\varepsilon_0^4).
\]
Equation \eqref{eq:energy-decay-implication} with exponent $\frac12$
therefore yields
\begin{equation}\label{eq:half-order-decay}
N(t)\leq C_s(1+t)^{-1/4},
\end{equation}
where $C_s$ depends only on the parameters and initial norms listed
in Theorem~\ref{thm:decay}. This proves the desired decay when
$s=\frac12$.
For $\frac12<s<1$, we use \eqref{eq:half-order-decay} and
H\"older's inequality in time to obtain
\[
\begin{aligned}
&\int_0^\infty N(t)^{2s+1}D(t)^{\frac32-s}\dd t
\leq
\left(\int_0^\infty D(t)\dd t\right)^{\frac32-s}
\left(\int_0^\infty
N(t)^{(2s+1)/(s-\frac12)}\dd t\right)^{s-\frac12}
\\
&<\infty,
\end{aligned}
\]
since $(2s+1)/(4(s-\frac12))>1$ for $\frac12<s<1$.
Equation \eqref{eq:negative-energy-integral} proves that
$\mathcal L_s$ is uniformly bounded. We may thus apply
\eqref{eq:energy-decay-implication} with the given value of $s$.

Finally, Lemma~\ref{lem:poisson} and \eqref{eq:qt-estimates} imply
\[
\|\nabla\phi\|_{H^3}+\|q_t\|_{H^2}\leq CN(t).
\]
This proves \eqref{eq:solution-decay}. Integrating
\eqref{eq:decay-energy} from $t$ to infinity gives
\[
\int_t^\infty D(\tau)\dd\tau\leq C\mathcal E(t)
\leq C_s(1+t)^{-s}.
\]
Moreover, \eqref{eq:qt-estimates} and $N(t)\leq\delta\leq1$ give,
for almost every $t\geq0$,
\[
\|q_t\|_{H^2}^2
\leq C\|\nabla u\|_{H^2}^2+CN(t)^2\|q\|_{H^3}^2
\leq CD(t).
\]
This proves \eqref{eq:dissipation-tail}.
\end{proof}

\subsection{Decay in exterior domains}

\begin{proof}[Proof of Theorem~\ref{thm:ext-decay}]
Use the same energy $\mathcal E$ as in
Proposition~\ref{prop:ext-apriori}. For small $\varepsilon_0$ it
satisfies \eqref{eq:decay-energy}. The conservative equations
\eqref{eq:conservative-decay} and the expression for $S$ in
\eqref{eq:decay-stress} are unchanged. In particular,
\begin{equation}\label{eq:ext-momentum-bounds}
\begin{aligned}
&\|m\|+\|u\|+\|\nabla\phi\|\leq CN(t),
\\
&\|\nabla m\|+\|\nabla u\|+\|\nabla^2\phi\|
\leq CD(t)^{1/2}.
\end{aligned}
\end{equation}
For the first gradient, use
$\nabla m=(1+q)\nabla u+u\otimes\nabla q$ and
$\|u\|_{L^6}\|\nabla q\|_{L^3}
\leq CN(t)\|\nabla u\|$.

Choose $R>0$ so that the obstacle lies strictly inside $B_R$.
Let $\chi$ vanish on $B_R$ and equal one outside $B_{2R}$, with
all its derivatives supported in
$\mathcal A=\{R<|x|<2R\}\subset\Omega$. Set
\[
\begin{aligned}
&\phi_{\mathcal A}(t)=\frac1{|\mathcal A|}
\int_{\mathcal A}\phi(t,x)\dd x,\qquad
\widetilde\phi=\chi(\phi-\phi_{\mathcal A}),
\\
&\widetilde q=\Delta\widetilde\phi,\qquad
\widetilde m=\chi m.
\end{aligned}
\]
These quantities are extended by zero near the obstacle and are
regarded as functions on $\R^3$. Norms and inner products of the
localized quantities below are over $\R^3$ unless another domain
is indicated. Although $\widetilde\phi$ need not belong to $L^6$
after subtracting the local mean, only its derivatives are used.
They satisfy
\[
\begin{aligned}
&\nabla\widetilde\phi
=\chi\nabla\phi+(\phi-\phi_{\mathcal A})\nabla\chi,
\\
&\widetilde q
=\chi q+2\nabla\chi\cdot\nabla\phi
+(\Delta\chi)(\phi-\phi_{\mathcal A}).
\end{aligned}
\]
By Lemma~\ref{lem:ext-poisson} and local Poincar\'e,
\begin{equation}\label{eq:ext-cutoff-bounds}
\begin{aligned}
&\|\widetilde q\|_{H^1(\R^3)}
\leq C\|q\|_{H^1(\Omega)},\qquad
\|\nabla^2\widetilde\phi\|_{L^2(\R^3)}
\leq C\|q\|_{L^2(\Omega)},
\\
&\|\nabla\widetilde m\|_{L^2(\R^3)}
\leq C\|\nabla u\|_{L^2(\Omega)}.
\end{aligned}
\end{equation}

The localized continuity equation has a compactly supported error:
\begin{equation}\label{eq:ext-continuity-error}
\widetilde q_t+\Div\widetilde m
=\nabla\chi\cdot m
+2\nabla\chi\cdot\nabla\phi_t
+(\Delta\chi)(\phi_t-\partial_t\phi_{\mathcal A}).
\end{equation}
Its right-hand side has zero mean. Indeed,
$m+\nabla\phi_t=\PP m$, and the right-hand side equals
\[
\nabla\chi\cdot\PP m+
\Div\bigl((\phi_t-\partial_t\phi_{\mathcal A})\nabla\chi\bigr).
\]
The divergence term integrates to zero. For the other term, use
the orthogonality of $\PP m$ to $\nabla(1-\chi)$ in
\eqref{eq:projection-definition}; the function $1-\chi$ is
compactly supported in $\overline\Omega$ and belongs to
$\dot H^1(\Omega)$.

The divergence equation on the fixed annulus therefore provides
a vector field $R_0\in H^1_0(\mathcal A;\R^3)$ whose divergence
is the right-hand side of \eqref{eq:ext-continuity-error}.
We use the usual bounded right inverse of divergence on a smooth
bounded domain; see \cite{Galdi2011}. Extending $R_0$ by zero
gives
\begin{equation}\label{eq:ext-R0}
\widetilde q_t+\Div\widetilde m=\Div R_0,\qquad
\|R_0\|_{L^2(\R^3)}\leq CD(t)^{1/2}.
\end{equation}
For this bound, the local norms of $m$ are controlled by
$\|\nabla u\|$. Moreover,
\[
\|\nabla\phi_t\|_{L^2(\mathcal A)}
+\|\phi_t-\partial_t\phi_{\mathcal A}\|_{L^2(\mathcal A)}
\leq C\|q_t\|\leq CD(t)^{1/2}
\]
by \eqref{eq:ext-poisson-time} and local Poincar\'e.
The whole-space gradient projection, denoted by $\QQ_{\R^3}$,
then gives
\begin{equation}\label{eq:ext-cutoff-field-time}
\nabla\widetilde\phi_t
=-\QQ_{\R^3}(\widetilde m-R_0).
\end{equation}

The localized momentum equation is
\begin{equation}\label{eq:ext-conservative-cutoff}
\widetilde m_t+A\widetilde m+
\gamma\nabla\widetilde q-\nabla\widetilde\phi
=\Div(\chi S)+R_1,
\end{equation}
where
\[
\begin{aligned}
&R_1={}[A,\chi]m+
\gamma\nabla(\widetilde q-\chi q)+\gamma q\nabla\chi
\\
&\quad -(\phi-\phi_{\mathcal A})\nabla\chi-S\nabla\chi.
\end{aligned}
\]
This vector field also has support in $\mathcal A$.
The first four terms are bounded in $L^2$ by $CD(t)^{1/2}$,
using \eqref{eq:ext-local-sobolev} and
Lemma~\ref{lem:ext-poisson}. The products in $S$ satisfy
$\|S\|_{L^2(\Omega)}\leq CN(t)D(t)^{1/2}$.
Consequently,
\begin{equation}\label{eq:ext-R1}
\|R_1\|_{L^2(\R^3)}\leq CD(t)^{1/2}.
\end{equation}
Here and below $N(t)\leq\delta\leq1$.

Fix $0<s<\frac14$. We use the whole-space multiplier
$\Lambda=(-\Delta)^{1/2}$ and Lemma~\ref{lem:ext-negative}.

We retain the notation $\mathcal L_s$ for the localized
negative-order energy:
\[
\mathcal L_s(t)=\frac12\left(
\gamma\|\Lambda^{-s}\widetilde q\|^2
+\|\Lambda^{-s}\widetilde m\|^2
+\|\Lambda^{-s}\nabla\widetilde\phi\|^2\right).
\]
Apply $\Lambda^{-s}$ to
\eqref{eq:ext-R0}, \eqref{eq:ext-conservative-cutoff} and
\eqref{eq:ext-cutoff-field-time}. The pressure and electric
terms cancel by the same calculation as in the proof of
Theorem~\ref{thm:decay}, now on $\R^3$. Thus
\begin{equation}\label{eq:ext-negative-identity}
\begin{aligned}
&\frac{\dd}{\dd t}\mathcal L_s
+\mu\|\nabla\Lambda^{-s}\widetilde m\|^2
+(\mu+\lambda)\|\Div\Lambda^{-s}\widetilde m\|^2
\\
&\quad=-(\chi S,\nabla\Lambda^{-2s}\widetilde m)
+(R_1,\Lambda^{-2s}\widetilde m)
-\gamma(R_0,\nabla\Lambda^{-2s}\widetilde q)
\\
&\qquad\quad+(R_0,\Lambda^{-2s}\nabla\widetilde\phi).
\end{aligned}
\end{equation}
By \eqref{eq:ext-cutoff-bounds} and
\eqref{eq:ext-negative-local}--\eqref{eq:ext-negative-density},
the last three terms are bounded in absolute value by $C_sD(t)$.
This is where compact support of $R_0,R_1$ is used.

We next estimate the stress term.
The exterior Sobolev inequalities and
Lemma~\ref{lem:ext-poisson} give, for $6\leq r\leq\infty$,
\[
\|u\|_{L^r}+\|\nabla\phi\|_{L^r}
+\|q\|_{L^r}+\|\nabla u\|_{L^r}+\|\nabla q\|_{L^r}
\leq C_rD(t)^{1/2}.
\]
For $u$ and $\nabla\phi$, the $L^\infty$ bound follows by
applying the $W^{1,6}$ embedding and then the homogeneous
$L^6$ estimate to the function and its first derivatives.
The other bounds follow directly from $H^2$ embedding.
Since
\[
\frac{3+4s}{6}=\frac12+\frac{2s}{3},
\qquad \frac3{2s}>6,
\]
H\"older's inequality applied to the terms in
\eqref{eq:decay-stress} gives
\begin{equation}\label{eq:ext-stress-product}
\|\chi S\|_{L^{6/(3+4s)}(\R^3)}
\leq C_sN(t)D(t)^{1/2}.
\end{equation}
For example, use the $L^2$ norm of $m$ and the
$L^{3/(2s)}$ norm of $u$ for the convective term, and the
$L^2$ and $L^{3/(2s)}$ norms of $\nabla\phi$ for the electric
term. The pressure remainder is a bounded smooth coefficient
times $q^2$. For the viscous corrections, write
$\nabla(qu)=q\nabla u+u\otimes\nabla q$ and take the
undifferentiated factor in $L^2$.
Also,
\[
\|\nabla\Lambda^{-2s}\widetilde m\|_{L^{6/(3-4s)}(\R^3)}
\leq C_s\|\nabla\widetilde m\|
\leq C_sD(t)^{1/2}
\]
by the Hardy--Littlewood--Sobolev inequality.
Therefore the first term on the right of
\eqref{eq:ext-negative-identity} is bounded by $C_sN(t)D(t)$.
We have proved
\begin{equation}\label{eq:ext-negative-energy}
\frac{\dd}{\dd t}\mathcal L_s
+c\|\nabla\Lambda^{-s}\widetilde m\|^2
\leq C_s(1+N(t))D(t).
\end{equation}
Regularizing the negative multiplier by
$(\epsilon^2-\Delta)^{-s/2}$ justifies the calculation before
finiteness of the negative norms is known. The estimates above
are uniform in $\epsilon$, and monotone convergence gives
\[
\sup_{t\geq0}\mathcal L_s(t)
\leq\mathcal L_s(0)+C_s\varepsilon_0^2.
\]

We verify the initial negative norms using only the stated
$L^1$ assumptions. By localization and
Lemma~\ref{lem:ext-poisson},
\[
\|\widetilde m(0)\|_{L^1(\R^3)}
\leq C\|u_0\|_{L^1(\Omega)},\qquad
\|\widetilde q(0)\|_{L^1(\R^3)}
\leq\|q_0\|_{L^1(\Omega)}+C\varepsilon_0.
\]
The corresponding $L^2$ norms, and
$\|\nabla\widetilde\phi(0)\|_{L^2(\R^3)}$, are bounded by
$C\varepsilon_0$. Since
$\Delta\widetilde\phi(0)=\widetilde q(0)$,
\[
\widehat{\nabla\widetilde\phi}(0,\xi)
=-\frac{i\xi}{|\xi|^2}\widehat{\widetilde q}(0,\xi)
\quad\text{for almost every }\xi\ne0.
\]
It follows that
\[
\int_{|\xi|<1}|\xi|^{-2s}
|\widehat{\nabla\widetilde\phi}(0,\xi)|^2\dd\xi
\leq C\|\widetilde q(0)\|_{L^1}^2
\int_0^1r^{-2s}\dd r<\infty.
\]
The last integral is finite for $s<\frac12$, which includes
the stated range. The negative norms of $\widetilde q(0)$
and $\widetilde m(0)$ follow directly from $L^1\cap L^2$,
and their high frequencies are controlled in $L^2$.
Consequently,
\begin{equation}\label{eq:ext-negative-uniform}
\sup_{t\geq0}\mathcal L_s(t)
\leq C_s\bigl(
\varepsilon_0^2+\|q_0\|_{L^1(\Omega)}^2
+\|u_0\|_{L^1(\Omega)}^2\bigr).
\end{equation}

It remains to apply the same interpolation argument as in
Theorem~\ref{thm:decay}. Whole-space Fourier interpolation and
\eqref{eq:ext-cutoff-bounds} give
\[
\|\widetilde m\|^2+\|\nabla\widetilde\phi\|^2
\leq C_s\mathcal L_s(t)^{1/(1+s)}
D(t)^{s/(1+s)}
\leq C_sD(t)^{s/(1+s)}.
\]
Since $1/2\leq1+q\leq3/2$, localization back to $\Omega$ yields
\[
\|u\|^2+\|\nabla\phi\|^2
\leq C\bigl(
\|\widetilde m\|^2+\|\nabla\widetilde\phi\|^2
+\|\nabla u\|^2+\|q\|^2\bigr).
\]
The compact-region terms here are bounded both by $CD(t)$
and by $CN(t)^2\leq C\varepsilon_0^2$.
The same is true of the first three terms in
\[
N(t)^2\leq C\bigl(
\|q\|_{H^3}^2+\|\nabla u\|_{H^2}^2+\|\nabla u_t\|^2
+\|u\|^2+\|\nabla\phi\|^2\bigr),
\]
which follows from the momentum equation exactly as in the
proof of Theorem~\ref{thm:decay}.
For any nonnegative quantity bounded by both $CD(t)$ and
$C\varepsilon_0^2$, interpolation bounds it by
$C\varepsilon_0^{2/(1+s)}D(t)^{s/(1+s)}$.
Thus
\[
\mathcal E(t)\leq C_sD(t)^{s/(1+s)},\qquad
\frac{\dd}{\dd t}\mathcal E+c_s\mathcal E^{1+1/s}\leq0.
\]
Integration proves \eqref{eq:ext-solution-decay}, since
$\mathcal E(t)\leq C_s(1+t)^{-s}$ and
$N(t)\leq C_s(1+t)^{-s/2}$.
Finally, \eqref{eq:qt-estimates} and
Lemma~\ref{lem:ext-poisson} give
$\|q_t\|_{H^2}+\|\nabla\phi\|_{H^3}\leq CN(t)$ and
$\|q_t\|_{H^2}^2\leq CD(t)$.
Integrating the energy inequality from $t$ to infinity proves
\eqref{eq:ext-dissipation-tail} and completes the proof.
\end{proof}

\section*{Acknowledgements}
This research is supported by the National Natural Science Foundation
of China (Nos. 12471198 and 12431018).


\begin{thebibliography}{99}
\setlength{\itemsep}{2pt}
\interlinepenalty=10000

\bibitem{BieWangYao2017}
Q. Y. Bie, Q. R. Wang and Z.-A. Yao,
Optimal decay rate for the compressible Navier--Stokes--Poisson system in
the critical $L^p$ framework,
\emph{J. Differential Equations} \textbf{263} (2017), 8391--8417.

\bibitem{ChenHuangShi2024}
Y. Z. Chen, B. Huang and X. D. Shi,
Global well-posedness of classical solutions to the compressible
Navier--Stokes--Poisson equations with slip boundary conditions in 3D
bounded domains, \emph{J. Math. Fluid Mech.} \textbf{26} (2024), Article 41.

\bibitem{ChikamiDanchin2017}
N. Chikami and R. Danchin,
On the global existence and time decay estimates in critical spaces for
the Navier--Stokes--Poisson system,
\emph{Math. Nachr.} \textbf{290} (2017), 1939--1970.

\bibitem{Danchin2000}
R. Danchin,
Global existence in critical spaces for compressible Navier--Stokes equations,
\emph{Invent. Math.} \textbf{141} (2000), 579--614.

\bibitem{Deckelnick1992}
K. Deckelnick,
Decay estimates for the compressible Navier--Stokes equations in unbounded
domains, \emph{Math. Z.} \textbf{209} (1992), 115--130.

\bibitem{Donatelli2003}
D. Donatelli,
Local and global existence for the coupled Navier--Stokes--Poisson problem,
\emph{Quart. Appl. Math.} \textbf{61} (2003), 345--361.

\bibitem{DuLiu2025}
Y. Z. Du and H. R. Liu,
Existence and nonlinear stability of steady state solutions of the 3-D
Navier--Stokes--Poisson equations with non-flat doping profile of large
variations in exterior domains,
\emph{J. Math. Phys.} \textbf{66} (2025), Article 081511.

\bibitem{FeireislNovotnyPetzeltova2001}
E. Feireisl, A. Novotn\'y and H. Petzeltov\'a,
On the existence of globally defined weak solutions to the Navier--Stokes
equations, \emph{J. Math. Fluid Mech.} \textbf{3} (2001), 358--392.

\bibitem{Galdi2011}
G. P. Galdi,
\emph{An Introduction to the Mathematical Theory of the Navier--Stokes
Equations: Steady-State Problems}, second edition, Springer, New York, 2011.

\bibitem{GaoPengWeiYao2026}
J. C. Gao, L. Y. Peng, Y. C. Wei and Z.-A. Yao,
Global uniform regularity and vanishing viscosity limit for the
compressible Navier--Stokes--Poisson equations with damping under
slip boundary condition,
\emph{J. Differential Equations} \textbf{460} (2026), Article 114067.

\bibitem{HaoLi2009}
C. C. Hao and H. L. Li,
Global existence for compressible Navier--Stokes--Poisson equations in
three and higher dimensions,
\emph{J. Differential Equations} \textbf{246} (2009), 4791--4812.

\bibitem{Hoff1995}
D. Hoff,
Global solutions of the Navier--Stokes equations for multidimensional
compressible flow with discontinuous initial data,
\emph{J. Differential Equations} \textbf{120} (1995), 215--254.

\bibitem{HoffZumbrun1995}
D. Hoff and K. Zumbrun,
Multi-dimensional diffusion waves for the Navier--Stokes equations of
compressible flow,
\emph{Indiana Univ. Math. J.} \textbf{44} (1995), 603--676.

\bibitem{HsiaoLi2010}
L. Hsiao and H. L. Li,
Compressible Navier--Stokes--Poisson equations,
\emph{Acta Math. Sci. Ser. B (Engl. Ed.)} \textbf{30} (2010), 1937--1948.

\bibitem{HuangLiXin2012}
X. D. Huang, J. Li and Z. P. Xin,
Global well-posedness of classical solutions with large oscillations and
vacuum to the three-dimensional isentropic compressible Navier--Stokes
equations, \emph{Comm. Pure Appl. Math.} \textbf{65} (2012), 549--585.

\bibitem{Jiang1996}
S. Jiang,
Global spherically symmetric solutions to the equations of a viscous
polytropic ideal gas in an exterior domain,
\emph{Comm. Math. Phys.} \textbf{178} (1996), 339--374.

\bibitem{KageiKobayashi2002}
Y. Kagei and T. Kobayashi,
On large-time behavior of solutions to the compressible Navier--Stokes
equations in the half space in $\mathbb{R}^{3}$,
\emph{Arch. Ration. Mech. Anal.} \textbf{165} (2002), 89--159.

\bibitem{KageiKobayashi2005}
Y. Kagei and T. Kobayashi,
Asymptotic behavior of solutions of the compressible Navier--Stokes
equations on the half space,
\emph{Arch. Ration. Mech. Anal.} \textbf{177} (2005), 231--330.

\bibitem{Kobayashi2002}
T. Kobayashi,
Some estimates of solutions for the equations of motion of compressible
viscous fluid in the three-dimensional exterior domain,
\emph{J. Differential Equations} \textbf{184} (2002), 587--619.

\bibitem{KobayashiShibata1999}
T. Kobayashi and Y. Shibata,
Decay estimates of solutions for the equations of motion of compressible
viscous and heat-conductive gases in an exterior domain in $\mathbb R^3$,
\emph{Comm. Math. Phys.} \textbf{200} (1999), 621--659.

\bibitem{KobayashiSuzuki2008}
T. Kobayashi and T. Suzuki,
Weak solutions to the Navier--Stokes--Poisson equation,
\emph{Adv. Math. Sci. Appl.} \textbf{18} (2008), 141--168.

\bibitem{LiMatsumuraZhang2010}
H. L. Li, A. Matsumura and G. J. Zhang,
Optimal decay rate of the compressible Navier--Stokes--Poisson system in
$\mathbb{R}^{3}$,
\emph{Arch. Ration. Mech. Anal.} \textbf{196} (2010), 681--713.

\bibitem{LiuLuoZhong2020}
H. R. Liu, T. Luo and H. Zhong,
Global solutions to compressible Navier--Stokes--Poisson and Euler--Poisson
equations of plasma on exterior domains,
\emph{J. Differential Equations} \textbf{269} (2020), 9936--10001.

\bibitem{LiuZhong2021}
H. R. Liu and H. Zhong,
Global solutions to the initial boundary problem of 3-D compressible
Navier--Stokes--Poisson on bounded domains,
\emph{Z. Angew. Math. Phys.} \textbf{72} (2021), Article 78.

\bibitem{MatsumuraNishida1979}
A. Matsumura and T. Nishida,
The initial value problem for the equations of motion of compressible
viscous and heat-conductive fluids,
\emph{Proc. Japan Acad. Ser. A Math. Sci.} \textbf{55} (1979), 337--342.

\bibitem{MatsumuraNishida1980}
A. Matsumura and T. Nishida,
The initial value problem for the equations of motion of viscous and
heat-conductive gases,
\emph{J. Math. Kyoto Univ.} \textbf{20} (1980), 67--104.

\bibitem{MatsumuraNishida1983}
A. Matsumura and T. Nishida,
Initial-boundary value problems for the equations of motion of compressible
viscous and heat-conductive fluids,
\emph{Comm. Math. Phys.} \textbf{89} (1983), 445--464.

\bibitem{ShenWang2026a}
R. Shen and Y. Wang,
Global well-posedness of the compressible electrically conducting
viscoelastic fluids subject to the Coulomb force in the half space,
\emph{J. Differential Equations} \textbf{451} (2026), Article 113766.

\bibitem{ShenWang2026b}
R. Shen and Y. Wang,
On the half-space problem of the compressible heat-conducting
electro-viscoelastic system,
\emph{J. Differential Equations} \textbf{474} (2026), Article 114448.

\bibitem{ShenWangWu2025}
R. Shen, Y. Wang and Y. S. Wu,
Global well-posedness of the compressible elastic
Navier--Stokes--Poisson equations in half-spaces,
\emph{J. Math. Phys.} \textbf{66} (2025), Article 011514.

\bibitem{TanWangWang2015}
Z. Tan, Y. J. Wang and Y. Wang,
Stability of steady states of the Navier--Stokes--Poisson equations with
non-flat doping profile,
\emph{SIAM J. Math. Anal.} \textbf{47} (2015), 179--209.

\bibitem{TanYangZhaoZou2013}
Z. Tan, T. Yang, H. J. Zhao and Q. Y. Zou,
Global solutions to the one-dimensional compressible Navier--Stokes--Poisson
equations with large data,
\emph{SIAM J. Math. Anal.} \textbf{45} (2013), 547--571.

\bibitem{Wang2012}
Y. J. Wang,
Decay of the Navier--Stokes--Poisson equations,
\emph{J. Differential Equations} \textbf{253} (2012), 273--297.

\bibitem{WangWu2010}
W. K. Wang and Z. G. Wu,
Pointwise estimates of solution for the Navier--Stokes--Poisson equations
in multi-dimensions,
\emph{J. Differential Equations} \textbf{248} (2010), 1617--1636.

\bibitem{WuWang2023}
W. P. Wu and Y. Wang,
On the half-space or exterior problems of the 3D compressible elastic
Navier--Stokes--Poisson equations,
\emph{SIAM J. Math. Anal.} \textbf{55} (2023), 2996--3043.

\bibitem{ZhangLiZhu2011}
G. J. Zhang, H. L. Li and C. J. Zhu,
Optimal decay rate of the non-isentropic compressible
Navier--Stokes--Poisson system in $\mathbb R^3$,
\emph{J. Differential Equations} \textbf{250} (2011), 866--891.

\end{thebibliography}
\end{document}